\documentclass[10pt,leqno]{amsart}
\usepackage{amssymb,amsthm}
\usepackage{amsmath}
\usepackage{setspace}
\usepackage{dcpic,pictexwd}
\usepackage[all]{xy}
\usepackage[pdftex]{graphicx}
\usepackage{mathrsfs}
\usepackage[pdftex]{hyperref}
\usepackage{bbm}
\usepackage{subcaption}
\usepackage{tikz}
\usetikzlibrary{cd}

\theoremstyle{definition}
 \newtheorem{definition}{Definition}[section]

\theoremstyle{plain}
 \newtheorem{proposition}[definition]{Proposition}

\theoremstyle{plain}
 \newtheorem{theorem}[definition]{Theorem}
 
\theoremstyle{definition}

\theoremstyle{plain}
 \newtheorem{lemma}[definition]{Lemma}

\theoremstyle{plain}

\theoremstyle{remark}
 \newtheorem{remark}[definition]{Remark}

\theoremstyle{definition}

\theoremstyle{plain}

\newcommand{\Ext}{\mathrm{Ext}}
\newcommand{\End}{\mathrm{End}}

\newcommand{\Mat}{\mathrm{Mat}}
\newcommand{\Hom}{\mathrm{Hom}}
\newcommand{\SHom}{\underline{\Hom}}

\newcommand{\Ca}{\mathcal{C}}
\newcommand{\Fun}{\mathrm{F}}

\newcommand{\Def}{\mathrm{Def}}
\newcommand{\Sets}{\mathrm{Sets}}

\newcommand{\Z}{\mathbb{Z}}
\newcommand{\SEnd}{\underline{\End}}
\newcommand{\A}{\Lambda}

\newcommand{\surjection}{\twoheadrightarrow}
\newcommand{\injection}{\hookrightarrow}

\newcommand{\m}{\mathfrak{m}}
\renewcommand{\k}{\Bbbk}
\renewcommand{\1}{\mathbbm{1}}
\newcommand{\invlim}{\varprojlim}

\newcommand{\M}{\widehat{M}}
\newcommand{\N}{\widehat{N}}

\newcommand{\V}{\widehat{V}}
\newcommand{\W}{\widehat{W}}
\renewcommand{\P}{\widehat{P}}
\newcommand{\Q}{\widehat{Q}}
\newcommand{\Ar}{\widehat{\Lambda}}
\newcommand{\h}{\widehat{h}}
\newcommand{\g}{\widehat{g}}

\hypersetup{
    bookmarks=true,         % show bookmarks bar?
    unicode=false,          % non-Latin characters in AcrobatÃs bookmarks
    pdftoolbar=true,        % show AcrobatÃs toolbar?
    pdfmenubar=true,        % show AcrobatÃs menu?
    pdffitwindow=false,     % window fit to page when opened
    pdfstartview={FitH},    % fits the width of the page to the window
    pdfnewwindow=true,      % links in new window
    colorlinks=true,       % false: boxed links; true: colored links
    linkcolor=red,          % color of internal links
    citecolor=blue,        % color of links to bibliography
    filecolor=magenta,      % color of file links
    urlcolor=blue           % color of external links
}

\title[On a deformation theory of modules over repetitive algebras]{On a deformation theory of finite dimensional modules over repetitive algebras} 
\thanks{This research was partially supported by CODI (Universidad de Antioquia, UdeA), COLCIENCIAS (Convocatoria Doctorados Nacionales 2016, N\'umero 757), and the Office of Academic Affairs at the Valdosta State University.}

\author[Fonce-Camacho]{Adriana Fonce-Camacho}
\address{Instituto de Matem\'aticas, Universidad de Antioquia, Medell{\'\i}n, Antioquia, Colombia}
\email{adrianam.fonce@udea.edu.co}
\author[Giraldo]{Hern\'an Giraldo}
\address{Instituto de Matem\'aticas, Universidad de Antioquia, Medell{\'\i}n, Antioquia, Colombia}
\email{hernan.giraldo@udea.edu.co}
\author[Rizzo]{Pedro Rizzo}
\address{Instituto de Matem\'aticas, Universidad de Antioquia, Medell{\'\i}n, Antioquia, Colombia}
\email{pedro.hernandez@udea.edu.co}
\author[V\'elez-Marulanda]{Jos\'e A. V\'elez-Marulanda (Corresponding author)}
\address{Department of Mathematics, Valdosta State University, Valdosta, GA,  United States of America}
\email{javelezmarulanda@valdosta.edu (Corresponding author)}

\begin{document}
\renewcommand{\labelenumi}{\textup{(\roman{enumi})}}
\renewcommand{\labelenumii}{\textup{(\roman{enumi}.\alph{enumii})}}
\numberwithin{equation}{section}

\begin{abstract}
Let $\A$ be a basic finite dimensional algebra over an algebraically closed field $\k$, and let $\Ar$ be the repetitive algebra of $\A$.  In this article, we prove that if $\V$ is a left $\Ar$-module with finite dimension over $\k$, then $\V$ has a well-defined versal deformation ring $R(\Ar,\V)$, which is a local complete Noetherian commutative $\k$-algebra whose residue field is also isomorphic to $\k$. We also prove that $R(\Ar, \V)$ is universal provided that $\SEnd_{\Ar}(\V)=\k$ and that in this situation, $R(\Ar, \V)$ is stable after taking syzygies. We apply the obtained results to finite dimensional modules over the repetitive algebra of the $2$-Kronecker algebra, which provides an alternative approach to the deformation theory of objects in the bounded derived category of coherent sheaves over $\mathbb{P}^1_\k$. 
\end{abstract}
\maketitle

\section{Introduction}

In this article, we assume that $\k$ is an algebraically closed field. If $f:A\to B$ and $g:B\to C$ are morphisms in a given $\k$-category $\mathscr{C}$, then the composition of $f$ with $g$ is denoted by $g\circ f$, i.e., the usual composition of morphisms. Modules considered in this article are mostly assumed to be finite dimensional over $\k$, and unless explicitly stated, they will be from the left side. If $\Gamma$ is a (not necessarily finite dimensional) $\k$-algebra, then we denote by $\Gamma$-mod the abelian category of finitely generated left $\Gamma$-modules and by $\Gamma$-\underline{mod} its stable category, i.e., the objects of $\Gamma$-\underline{mod} are the same as those in $\Gamma$-mod, and two morphisms $f, g: X\to Y$ between object in $\Gamma$-\underline{mod} are identified if and only if $f-g$ factors through a finitely generated projective left $\Gamma$-module. If $X$ in $\Gamma$-mod has a projective cover, namely $\psi_X: P(X)\to X$, then we denote by $\Omega_\Gamma X$, the kernel of $\psi_X$. Note that in this situation, projective covers are unique up to isomorphism. If $\Gamma$-mod has almost split sequences (in the sense of \cite[Chap. 5]{auslander}), then for all indecomposable left $\Gamma$-modules $X$, we denote by $\tau_\Gamma X$ the {\it Auslander-Reiten translation} of $X$. 

Let $\A$ be a fixed basic finite dimensional $\k$-algebra, and let $\Ar$ be the repetitive algebra of $\A$ (in the sense of \cite{hughes-wash}). It follows from \cite[Chap II, \S 2.2]{happel} that $\Ar$-mod is a {\it Frobenius category} in the sense of \cite[Chap. I, \S 2]{happel}, i.e., $\Ar$-mod is an exact category in the sense of \cite{quillen} and which has enough projective as well as injective objects, and these classes of objects coincide. By \cite{hughes-wash}, the abelian category $\Ar$-mod has almost split sequences. Moreover, it follows by \cite[Chap. II, \S 2.2]{happel} that $\Ar$-\underline{mod} is a triangulated category (in the sense of \cite{verdier}). 

Let $\mathcal{D}^b(\A\textup{-mod})$ be the bounded derived category of $\A$, which is also a triangulated category. Due to a fundamental result by D. Happel, there exists a full and faithful exact functor of triangulated categories  $\mu: \mathcal{D}^b(\A\textup{-mod})\to \Ar\textup{-\underline{mod}}$, such that $\mu$ extends the identity functor on $\A$-mod, where $\A$-mod is embedded in $\mathcal{D}^b(\A\textup{-mod})$ (resp. $\Ar$-\underline{mod}) as complexes (resp. modules) concentrated in degree zero. Moreover, $\mu$ is an equivalence if and only if $\A$ has finite global dimension (see \cite[\S 2.3]{happel2}). 

Let $\mathcal{D}^b(\mathrm{coh}\,\mathbb{P}^n_\k)$ denote the bounded derived category of coherent sheaves over the projective $n$-space $\mathbb{P}^n_\k$.  In \cite{bernstein,beilinson}, I. N. Bernstein et al. and A. A. Beilinson constructed a $\k$-algebra $\A_n$ of finite global dimension such that there exist equivalences of triangulated categories

\begin{align}\label{triangequiv}
\mathcal{D}^b(\mathrm{coh}\,\mathbb{P}^n_\k)\cong \mathcal{D}^b(\A_n\textup{-mod})&&\text{ and }&& \mathcal{D}^b(\mathrm{coh}\,\mathbb{P}^n_\k)\cong \Ar_n\textup{-\underline{mod}}.
\end{align}

The precise relationship between the equivalences in (\ref{triangequiv}) and $\mu$ were completely determined by P. Dowbor and H. Meltzer in \cite{dowbor}. As a consequence, one can pass freely between the triangulated categories $\mathcal{D}^b(\mathrm{coh}\,\mathbb{P}^n_\k)$, $\mathcal{D}^b(\A_n\textup{-mod})$ and $ \Ar_n\textup{-\underline{mod}}$. The quiver with relations of the $\k$-algebra $\A_n$ is available in e.g. \cite[Chap. II, \S 5]{happel}.  In particular, when $n=1$, we have that $\A_1$ is the $2$-Kronecker algebra whose quiver (with empty set of relations) is given as follows:
\begin{equation}\label{kronecker}
Q=\xymatrix@=20pt{
\underset{1}{\bullet}\ar@/^/[rr]^{\alpha}\ar@/_/[rr]_{\beta}&&\underset{2}{\bullet}
}
\end{equation}

Moreover, by using \cite[Theorem on  pg. 428]{schroer}, the quiver with relations $(\widehat{Q}, \widehat{\rho})$ of $\Ar_1$ is given as follows:

\begin{equation}\label{repquiver}
\widehat{Q} =  
\begin{tikzcd}[column sep = 60, row sep = 60]
\cdots&\overset{1_{z-1}}{\bullet}\arrow[d, bend left=15,"\alpha_{z-1}"]\arrow[d, bend right=15,"\beta_{z-1}"']&\overset{1_z}{\bullet}\arrow[d, bend left=15,"\alpha_z"]\arrow[d, bend right=15,"\beta_z"']&\overset{1_{z+1}}{\bullet}\arrow[d, bend left=15,"\alpha_{z+1}"]\arrow[d, bend right=15,"\beta_{z+1}"']&\\
&\underset{2_{z-1}}{\bullet}\arrow[ul, bend left=10,"\alpha_{z-1}^\ast", near start]\arrow[ul, bend right=10,"\beta_{z-1}^\ast"',near end]&\underset{2_z}{\bullet}\arrow[ul, bend left=10,"\alpha_z^\ast", near start]\arrow[ul, bend right=10,"\beta_z^\ast"',near end]&\underset{2_{z+1}}{\bullet}\arrow[ul, bend left=10,"\alpha_{z+1}^\ast", near start]\arrow[ul, bend right=10,"\beta_{z+1}^\ast"',near end]&\cdots\arrow[ul, bend left=10,"\alpha_{z+2}^\ast", near start]\arrow[ul, bend right=10,"\beta_{z+2}^\ast"',near end]
\end{tikzcd}
\end{equation}

\begin{equation*}
\widehat{\rho}=\{\beta_z^\ast\alpha_z,\alpha_{z-1}\beta_z^\ast,\alpha_z^\ast\beta_z, \beta_{z-1}\alpha_z^\ast,\alpha^\ast_z\alpha_z-\beta^\ast_z\beta_z, \alpha_{z-1}\alpha^\ast_z-\beta_{z-1}\beta^\ast_z: z\in \Z\}.
\end{equation*}

%\arrow[ul, bend left=15,"\alpha_{z-1}"]\arrow[ul, bend right=15,"\beta_{z-1}"']
Observe that we read the paths in $\Q$ from the right to the left. Note as well that $\Ar_1$ is a special biserial $\k$-algebra (as introduced in \cite{wald}). This implies that the indecomposable objects and the irreducible morphisms in $\Ar_1$-\underline{mod} can be described combinatorially by using so-called string and bands for $\Ar_1$; the corresponding objects are called string and band $\Ar_1$-modules  (see \cite{buri}). This also implies the $\k$-vector space $\Hom_{\Ar_1}(\V, \V')$, where $\V$ and $\V'$ are indecomposable finite dimensional left $\Ar_1$-modules, can be described completely by using the results in \cite{krause}. 

On the other hand, in \cite[Prop. 2.1]{blehervelez}, F. M. Bleher and the fourth author developed a deformation theory of finitely generated modules over finite dimensional $\k$-algebras based on the deformation theory of Galois representation developed by B. Mazur in \cite{mazur0,mazur}. More precisely, they proved that if $V$ is a finitely generated left $\A$-module, then $V$ has a well-defined versal deformation ring $R(\A,V)$ which is a local complete Noetherian commutative $\k$-algebra whose residue field is also isomorphic to $\k$, and that $R(\A,V)$ is universal provided that the endomorphism ring of $V$ is isomorphic to $\k$. Moreover, they also proved in \cite[Thm.1.1]{blehervelez} that if $\A$ is further a Frobenius algebra and $V$ has stable endomorphism ring $\SEnd_\A(V)$ isomorphic to $\k$, then $R(\A,V)$ is universal and that the versal deformation rings $R(\A,V)$, $R(\A, \Omega_\A V)$ and $R(\A, V\oplus P)$ are isomorphic, where $P$ is an arbitrary finitely generated projective left $\A$-module. These results together with the combinatorial tools provided by special biserial $\k$-algebras have been used in order to study the deformation theory of finite dimensional modules over a number of finite dimensional self-injective $\k$-algebras (see e.g. \cite{bleher9,bleher15,calderon-giraldo-rueda-velez} and \cite{velez}).  

Although a deformation theory of $\mathrm{coh}\,\mathbb{P}^1_\k$ under the language of algebraic geometry has been studied by several authors (see e.g. \cite{nitsure} and its references), we aim to provide through this paper an alternative approach that uses the aforementioned machinery from representation theory of associative algebras together with the relationship between the equivalences (\ref{triangequiv}) and the Happel's functor $\mu$, such that we are able to arrive to a better understanding of the deformation theory of objects in $\mathcal{D}^b(\mathrm{coh}\,\mathbb{P}^1_\k)$.  In order to do so, we extend the deformation theory developed in \cite{blehervelez} to that of finite dimensional left $\Ar$-modules.  More precisely, our main result is the following.

\begin{theorem}\label{thm1}
Let $\V$ be a left $\Ar$-module with finite dimension over $\k$. Then $\V$ always has a well-defined versal deformation ring $R(\Ar,V)$. Assume further that $\SEnd_{\Ar}(\V)=\k$. Then we have the following.
\begin{enumerate}
\item The versal deformation ring $R(\Ar,\V)$ is universal, and for all finitely dimensional projective $\Ar$-modules $\P$, the versal deformation ring $R(\Ar,\V\oplus\P)$ is also universal and isomorphic to $R(\Ar,\V)$. 
\item The versal deformation ring $R(\Ar,\Omega_{\Ar}\V)$ is also universal and isomorphic to $R(\Ar,\V)$.
\item If $\V$ is indecomposable, then the versal deformation ring $R(\Ar, \tau_{\Ar}\V)$ is universal and isomorphic to $R(\Ar, \V)$.     
\end{enumerate}
\end{theorem}

As a consequence of Theorem \ref{thm1}, we obtain the following result, which by our discussion above, it provides some results that may be useful in order to arrive to a better understanding of the deformation theory of $\mathcal{D}^b(\mathrm{coh}\,\mathbb{P}^1_\k)$.

\begin{theorem}\label{thm2}
Let $\V$ be a finite dimensional string $\Ar_1$-module, where $\Ar_1 = \k \widehat{Q}/\langle \widehat{\rho}\rangle$ is as in (\ref{repquiver}). Let  $\Gamma_S(\Ar_1)$  be the stable Auslander-Reiten quiver of $\Ar_1$ and $\mathfrak{C}$ be the connected component of $\Gamma_S(\Ar_1)$ containing $\V$. 
\begin{enumerate}
\item The component $\mathfrak{C}$ of $\Gamma_S(\Ar_1)$ contains either a simple $\Ar_1$-module, or a string $\Ar_1$-module $\M[C]$ where $C$ is an arrow of $\widehat{Q}$. 
%In both situations, $\SEnd_{\Ar_1}(\V)=\k$.
\item If $\mathfrak{C}$ contains a simple $\Ar_1$-module, then $\End_{\Ar_1}(\V)=\k$ and $R(\Ar_1,\V)\cong \k$. 
\item If $\mathfrak{C}$ contains a string $\Ar_1$-module $\M[C]$ with $C$ an arrow of $\Q$, then $\SEnd_{\Ar_1}(\V)=\k$ if and only if $\V$ is isomorphic to $\M[C]$. In this situation, $R(\Ar_1,\V)\cong \k[\![t]\!]$. 
\end{enumerate}  
\end{theorem}

This article is organized as follows. In \S \ref{sec2}, we review some preliminary results concerning the repetitive algebras of finite dimensional algebras; our main source is \cite{happel}. In \S \ref{sec3}, following \cite{blehervelez}, we introduce lifts and deformations of finite dimensional modules over repetitive algebras and prove Theorem \ref{thm1}.  Finally, in  \S \ref{sec4}, we review some important aspects of the representation theory of the repetitive algebra of the $2$-Kronecker algebra as in (\ref{repquiver}) and prove Theorem \ref{thm2}. 

We refer the reader to look at e.g \cite{auslander} or \cite{assem3} in order to get further information regarding concepts from the representation theory of associative algebras such as irreducible morphisms, almost split sequences and Auslander-Reiten quivers. For an extended historical background on the deformation theory of modules and for an approach to the deformation theory of complexes over finite dimensional algebras  we invite the reader to look at \cite{blehervelez2}.

This article constitutes the doctoral dissertation of the first author under the supervision of the other three authors.  

\section{Preliminary results}\label{sec2}

Recall that throughout this article, we assume that $\k$ is a fixed algebraically closed field of arbitrary characteristic. 

Let $\A$ be a fixed basic finite dimensional $\k$-algebra.  Recall that $\A$-mod denotes the abelian category of finitely generated left $\A$-modules.  We denote by $D(-)=\Hom_\k(-,\k)$ the standard $\k$-duality on $\A$-mod. In particular $D\A$ is the minimal injective cogenerator of $\A$-mod. Note that $D\A$ is also a finitely generated $\A$-$\A$-bimodule in the following way. For all $\varphi \in D\A$, and $a', a''\in \A$, $a'\varphi a''$ is the $\k$-linear morphism that sends each $a\in\A$ to $\varphi(a''aa')$.    

\subsection{Modules over repetitive algebras}\label{repetitivealgebra}

In the following, we recall that definition of the {\it repetitive algebra} $\Ar$ of $\A$:
\begin{enumerate}
\item the underlying $\k$-vector space of $\Ar$  is given by $\Ar=\left(\bigoplus_{i\in \Z} \A\right)\oplus \left(\bigoplus_{i\in \Z} D\A\right)$;
\item the elements of $\Ar$ are denoted by $(a_i, \varphi_i)_i$, where $a_i\in \A$, $\varphi_i\in D\A$ and almost all of the $a_i, \varphi_i$ are zero;  
\item the product of two elements $(a_i, \varphi_i)_i$ and $(b_i,\psi_i)_i$ in $\Ar$ is defined as 
\[(a_i,\varphi_i)_i\cdot (b_i,\psi_i)_i= (a_ib_i, a_{i+1}\psi_i+\varphi_ib_i)_i.\]
\end{enumerate}

Although the repetitive algebra of a finite dimensional $\k$-algebra was originally introduced by  D. Hughes and J. Waschb\"usch in \cite{hughes-wash}, in this article, we mostly follow the notation in \cite[Chap. II, \S 2]{happel}. We refer the reader to \cite[Chap. II, \S 2]{happel} for an interpretation of $\Ar$ as a doubly infinite matrix $\k$-algebra. 

The $\Ar$-modules are of the form $\V=(V_i,f_i)_{i\in \Z}$, where for each $i\in \Z$, $V_i$ is a finitely generated $\A$-module, all but finitely many being zero, and $f_i$ is the morphism of $\A$-modules $f_i:D\A\otimes_\A V_i\to V_{i+1}$ such that $f_{i+1}\circ(\mathrm{id}_{D\A}\otimes f_i)=0$, where $\mathrm{id}_{D\A}$ denotes the identity morphism on $D\A$.  Moreover,  we can visualize a $\Ar$-module $\V=(V_i,f_i)$ as follows:

\begin{equation}\label{tensorcomplx}
\xymatrix@=20pt{
\cdots\ar@{~>}[r]&V_{i-2}\ar@{~>}[r]^{f_{i-2}}&V_{i-1}\ar@{~>}[r]^{f_{i-1}}&V_i\ar@{~>}[r]^{f_i}&V_{i+1}\ar@{~>}[r]^{f_{i+1}}&V_{i+2}\ar@{~>}[r]&\cdots
}
\end{equation} 

A morphism $\h:\V\to \V'$ of $\Ar$-modules is a sequence $\h=(h_i)_{i\in \Z}$ of morphisms of $\A$-modules $h_i: V_i\to V_i'$ such that, for all $i\in \Z$, $h_{i+1}\circ f_i=f_i'\circ (\mathrm{id}_{D\A}\otimes h_i)$.
% the following diagram commutes.
%\begin{equation*}
%\xymatrix@=20pt{
%Q\otimes_\A M_i\ar[rr]^{f_i}\ar[dd]^{\mathrm{id}_Q\otimes h_i}&&M_{i+1}\ar[dd]^{h_{i+1}}\\\\
%Q\otimes_\A M_i'\ar[rr]^{f'_i}&&M'_{i+1}
%}
%\end{equation*}
By using the description (\ref{tensorcomplx}), we can also visualize morphisms $\h:\V\to \V'$ between $\Ar$-modules as follows:

\begin{equation}\label{tensorcomplxmorph1}
\xymatrix@=20pt{
\cdots\ar@{~>}[r]&V_{i-2}\ar@{~>}[r]^{f_{i-2}}\ar[d]^{h_{i-2}}&V_{i-1}\ar@{~>}[r]^{f_{i-1}}\ar[d]^{h_{i-1}}&V_i\ar@{~>}[r]^{f_i}\ar[d]^{h_i}&V_{i+1}\ar@{~>}[r]^{f_{i+1}}\ar[d]^{h_{i+1}}&V_{i+2}\ar@{~>}[r]\ar[d]^{h_{i+2}}&\cdots\\
\cdots\ar@{~>}[r]&V'_{i-2}\ar@{~>}[r]^{f'_{i-2}}&V'_{i-1}\ar@{~>}[r]^{f'_{i-1}}&V'_i\ar@{~>}[r]^{f'_i}&V'_{i+1}\ar@{~>}[r]^{f'_{i+1}}&V'_{i+2}\ar@{~>}[r]&\cdots
}
\end{equation} 

In the situation of (\ref{tensorcomplxmorph1}), we call each $h_i$ the {\it $i$-th component morphism} of $\h$.

\begin{remark}\label{rem1.1}
Let $\Ar$-mod denote the abelian category of modules over $\Ar$. 

\begin{enumerate}
\item It follows by \cite[Chap II, \S 2.2]{happel} that $\Ar$-mod is a {\it Frobenius category} in the sense of \cite[Chap. I, \S 2]{happel}, i.e., $\Ar$-mod is an exact category in the sense of \cite{quillen} which has enough projective as well as injective objects, and these classes of objects coincide. 

\item By \cite[Prop. 6]{giraldo} we have that the indecomposable projective-injective $\Ar$-modules can be represented as follows:
\begin{equation}\label{projindec}
\xymatrix@=20pt{
\cdots\ar@{~>}[r]&0\ar@{~>}[r]&\A\epsilon\ar@{~>}[r]^{f}&D(\epsilon\A)\ar@{~>}[r]&0\ar@{~>}[r]&\cdots,
}
\end{equation}
where $\epsilon$ is a primitive orthogonal idempotent element in $\A$ and $f$ is the natural isomorphism of $\A$-modules $f:D\A\otimes_\A\A\epsilon\to D(\epsilon\A)$.  Note in particular that this description implies that the indecomposable projective $\Ar$-modules are finite dimensional over $\k$. Recall that we denote by $\Ar$-\underline{mod} the stable category of $\Ar$-mod. For all objects $\V$ in $\Ar$-mod, we denote by $\Omega^{-1}_{\Ar}\V$ the first cosyzygy of $\V$, i.e., $\Omega^{-1}_{\Ar}\V$ is the cokernel of an injective $\Ar$-module hull $\V\to I(\V)$, which is unique up to isomorphism. It follows from \cite[Chap. I, \S 2.2]{happel} that $\Omega^{-1}_{\Ar}$ induces an automorphism of $\Ar$-\underline{mod}. We denote by $\Omega_{\Ar}$ the quasi-inverse of $\Omega_{\Ar}^{-1}$. Since $\Ar$-mod is a Frobenius category, it follows that $\V$ has a projective cover $P(\V)\xrightarrow{\widehat{\psi}} \V$, which is unique up to isomorphism and that  $\Omega_{\Ar}\V$ is the kernel of $\widehat{\psi}$. 

\item It follows from \cite[Chap. I, \S 2.6]{happel} that $\Ar$-\underline{mod} is a triangulated category in the sense of \cite{verdier} whose translation functor is $\Omega^{-1}_{\Ar}$. 

\item Let $T\A$ be the {\it trivial extension} of $\A$, i.e., $T\A=\A\oplus D\A$ as $\k$-vector spaces and the product is defined as 
\begin{equation*}
(a,\varphi)\cdot (b, \psi)=(ab,a\psi+\varphi b), 
\end{equation*} 
for all $a,b\in \A$ and $\varphi,\psi\in D\A$.
It was noted in \cite[Chap. II, \S2.4]{happel} that $T\A$ is a $\Z$-graded algebra, where the elements of $\A\oplus 0$ are the elements of degree $0$ and those in $0\oplus D\A$ are the elements of degree $1$.  Moreover, if $T\A^{\Z}$-mod is the category of finitely generated $\Z$-graded $T\A$-modules with morphisms of degree $0$, then the categories $\Ar$-mod and $T\A^{\Z}$-mod are equivalent. Moreover, by the remarks in \cite[Chap. VI, \S 2]{maclane}, it follows that if $\V=(V_i,f_i)_{i\in \Z}$, $\V'=(V'_i,f'_i)_{i\in \Z}$ and $\V''=(V''_i,f''_i)_{i\in \Z}$ are $\Ar$-modules, then 
$0\to \V\xrightarrow{\h} \V'\xrightarrow{\h'} \V''\to 0$ is exact in $\Ar$-mod if and only if $0\to V_i\xrightarrow{h_i}V'_i\xrightarrow{h'_i}V''_i\to 0$ is exact in $\A$-mod such that $h_{i+1}\circ f_i=f_i'\circ (\mathrm{id}_{D\A}\otimes h_i)$ and $h_{i+1}'\circ f_i'=f_i''\circ (\mathrm{id}_{D\A}\otimes h_i')$ for all $i\in \Z$.

\item There is an automorphism $\nu_{\Ar}: \Ar\textup{-\underline{mod}}\to\Ar\textup{-\underline{mod}}$ defined as follows: for all objects $\widehat{V}=(V_i,f_i)_{i\in \Z}$ as in (\ref{tensorcomplx}), $\nu_{\Ar}\V$ is defined as $(\nu_{\Ar}\V)_i= V_{i+1}$. It follows from the arguments in \cite[\S 2.3]{hughes-wash} that $\nu_{\Ar}\P = D\Ar\otimes_{\Ar}\P\cong D\Hom_{\Ar}(\P,\Ar)$ for all finite dimensional projective $\Ar$-modules $\P$ (cf. \cite[Prop. III.5.2]{skow3}). Moreover, by \cite[\S 2.5]{hughes-wash}, the abelian category $\Ar$-mod has almost split sequences, and consequently $\Ar$-\underline{mod} has Auslander-Reiten triangles in the sense of \cite[Chap. I, \S 4]{happel}. On the other hand, it is noted in \cite[pg. 1614]{happkellrei} that the translation $\tau_{\Ar}$ is $D\mathrm{Tr}$ (as in \cite[Chap. IV]{auslander}). Therefore, if $\V$ is an indecomposable non-projective $\Ar$-module, then we have $\tau_{\Ar}\V=\nu_{\Ar}\Omega^2_{\Ar}\V$. The proof of this fact can be obtained by using the above discussion concerning $\nu_{\Ar}$, by using that every projective $\Ar$-module is also injective, and by adapting the arguments in the proof of \cite[Prop. IV.3.7 a)]{auslander} to our context.

\item Let $\V$ be a finite dimensional $\Ar$-module. Since $\Ar$-mod is a Frobenius category, it follows that for all $n>1$, $\Ext_{\Ar}^n(\V,\P)=0$ for all finite dimensional projective $\Ar$-modules $\P$. Now if  $\W$ is also a finite dimensional $\Ar$-module, then by adapting the arguments in the proof of \cite[Lemma 5.1]{skart} to our context, we obtain an isomorphism of $\k$-vectors spaces $\Ext_{\Ar}^1(\V,\W)=\SHom_{\Ar}(\Omega_{\Ar}\V,\W)$.

\end{enumerate}
\end{remark}

\section{Lifts and deformations of modules over repetitive algebras}\label{sec3}

We denote by $\widehat{\Ca}$ the category of all complete local commutative Noetherian $\k$-algebras with residue field $\k$. In particular, the morphisms in $\widehat{\Ca}$ are continuous $\k$-algebra homomorphisms that induce the identity map on $\k$. For all  objects $R$ in $\widehat{\Ca}$, we denote by $R\A$ and $R\Ar$ the tensor product of $\k$-algebras $R\otimes_\k\A$ and $R\otimes_\k \Ar$, respectively.

Let $\V$ be a $\Ar$-module with finite dimension over $\k$, and let $R$ be a fixed but arbitrary object in $\widehat{\Ca}$. 

\begin{definition}{(cf. \cite[\S 2]{blehervelez})}\label{def3.1}

\begin{enumerate}
\item A {\it lift} $(\M,\widehat{\phi})$ of $\V$ over $R$ is a finitely generated $R\Ar$-module $\M$ which is free over $R$ together with an isomorphism of $\Ar$-modules $\widehat{\phi}: \k\otimes_R\M\to \V$.
\item Two lifts $(\M,\widehat{\phi})$ and $(\M', \widehat{\phi}')$ of $\V$ over $R$ are said to be {\it isomorphic} provided that there exists an isomorphism of $R\Ar$-modules $\widehat{\alpha}:\M\to \M'$ such that $\widehat{\phi}'\circ(\mathrm{id}_\k\otimes \widehat{\alpha})=\widehat{\phi}$. 
\item If $(\M,\widehat{\phi})$ is a lift of $\V$ over $R$, we denote by $[\M,\widehat{\phi}]$ its isomorphism class and call it a {\it deformation} of $\V$ over $R$. We denote by $\mathrm{Def}_{\Ar}(\V,R)$ the set of all deformations of $\V$ over $R$. 
\end{enumerate}
\end{definition}

\begin{remark}\label{rem3.2}
Let $D_R(R\A)$ be the $R$-dual $\Hom_R(R\A,R)=R\otimes_\k D\A$. Note that $D_R(R\A)$ is in particular an $R\A$-$R\A$-bimodule. We denote by $\pi_{D\A,R}: \k \otimes_RD_R(R\A)\to D\A$ the natural isomorphism of tensor products.

\begin{enumerate}
\item If $(\M,\widehat{\phi})$ is a lift of $\V$ over $R$, then after using the identification 
\begin{equation*}
\k\otimes_R(D_R(R\A)\otimes_{R\A}M_i)= (\k\otimes_RD_R(R\A))\otimes_\A(\k\otimes_RM_i), 
\end{equation*}
we have that $\M$ can be represented as 
\begin{equation}\label{tensorcomplxM}
\xymatrix@=20pt{
\cdots\ar@{~>}[r]&M_{i-2}\ar@{~>}[r]^{g_{i-2}}&M_{i-1}\ar@{~>}[r]^{g_{i-1}}&M_i\ar@{~>}[r]^{g_i}&M_{i+1}\ar@{~>}[r]^{g_{i+1}}&M_{i+2}\ar@{~>}[r]&\cdots,
}
\end{equation} 
where for each $i\in \Z$, $g_i: D_R(R\A)\otimes_{R\A}M_i\to M_{i+1}$ is a morphism of $R\A$-modules such that  $\phi_{i+1}\circ (\mathrm{id}_\k\otimes g_i)= f_i\circ (\pi_{D\A,R}\otimes_\A\phi_i)$.  Moreover, since $\M$ is free over $R$, it follows that for each $i\in \Z$, the $R\A$-module $M_i$ is free over $R$, and by Remark \ref{rem1.1} (iv), it follows $\phi_i:\k\otimes_RM_i\to V_i$ is an isomorphism of $\A$-modules. Therefore, for all $i\in \Z$,  the pair $(M_i, \phi_i)$ is a lift of $V_i$ over $R$ in the sense of \cite[\S 2]{blehervelez}. 

\item Let $(\M, \widehat{\phi})$ and $(\M', \widehat{\phi}')$ be isomorphic lifts of $\V$ over $R$, and let $\widehat{\alpha}: \M\to \M'$ be an isomorphism of $R\Ar$-modules such that $ \widehat{\phi}'\circ(\mathrm{id}_\k\otimes \widehat{\alpha})=\widehat{\phi}$. Then by using the description (\ref{tensorcomplxM}), we have that  $\widehat{\alpha}: \M\to \M'$ can be represented as 

\begin{equation}\label{tensorcomplxmorph}
\xymatrix@=20pt{
\cdots\ar@{~>}[r]&M_{i-2}\ar@{~>}[r]^{g_{i-2}}\ar[d]^{\alpha_{i-2}}&M_{i-1}\ar@{~>}[r]^{g_{i-1}}\ar[d]^{\alpha_{i-1}}&M_i\ar@{~>}[r]^{g_i}\ar[d]^{\alpha_i}&M_{i+1}\ar@{~>}[r]^{g_{i+1}}\ar[d]^{\alpha_{i+1}}&M_{i+2}\ar@{~>}[r]\ar[d]^{\alpha_{i+2}}&\cdots\\
\cdots\ar@{~>}[r]&M'_{i-2}\ar@{~>}[r]^{g'_{i-2}}&M'_{i-1}\ar@{~>}[r]^{g'_{i-1}}&M'_i\ar@{~>}[r]^{g'_i}&M'_{i+1}\ar@{~>}[r]^{g'_{i+1}}&M'_{i+2}\ar@{~>}[r]&\cdots
}
\end{equation} 
where for each $i\in \Z$, $\alpha_i: M_i\to M'_i$ is an isomorphism of $R\A$-modules such that $\alpha_{i+1}\circ g_i= g_i'\circ \alpha_i$ and $\phi'_i\circ (\mathrm{id}_\k\otimes \alpha_i)=\phi_i$. In particular, for all $i\in \Z$, $\alpha_i:M_i\to M_i'$ induces an isomorphism of the lifts $(M_i, \phi_i)$ and $(M_i', \phi_i')$ of $V_i$ over $R$ in the sense of \cite[\S 2]{blehervelez}.  
\end{enumerate}
\end{remark}

\begin{definition}{(cf. \cite[\S 2]{blehervelez})}
The {\it deformation functor} corresponding to $\V$ is the 
covariant functor $\widehat{\Fun}_{\Ar,\V}:\hat{\Ca}\to \Sets$ defined as follows:
\begin{enumerate}
\item  for all objects $R$ in $\hat{\Ca}$, define $\widehat{\Fun}_{\Ar,\V}(R)=\Def_
{\Ar}(\V,R)$; 
\item for all morphisms $\theta:R\to 
R'$ in $\hat{\Ca}$, 
let $\widehat{\Fun}_{\Ar,\V}(\theta):\Def_{\Ar}(\V,R)\to \Def_{\Ar}(\V,R')$ be defined as 
\begin{equation*}
\widehat{\Fun}_{\Ar,\V}(\theta)([\M,\widehat{\phi}])=[R'\otimes_{R,\theta}\M,\widehat{\phi}_\theta],
\end{equation*} 
where $\widehat{\phi}_\theta: \k\otimes_{R'}
(R'\otimes_{R,\theta}\M)\to \V$ is the composition of $\A$-module isomorphisms
\begin{equation*} 
\k\otimes_{R'}(R'\otimes_{R,\theta}\M)\cong \k\otimes_R\M\xrightarrow{\widehat{\phi}} V.
\end{equation*}
\end{enumerate}
\item We denote by $\Fun_{\Ar,\V}$ the restriction of $\widehat{\Fun}_{\Ar, \V}$ to the full subcategory of Artinian objects in $\widehat{\Ca}$. 
\end{definition} 
%\nocite{*}
\begin{remark}\label{rem3.3}

\begin{enumerate}
\item If $\V$ has exactly one non-zero term, say a $\A$-module $V$, then $\widehat{\Fun}_{\Ar,\V}$ coincides with the deformation functor $\widehat{\Fun}_{\A, V}$ in the sense of \cite[\S 2]{blehervelez}.
\item Since the $\Ar$-module $\V$ has finite $\k$-dimension, we can assume that $\V$ is of the form 
\begin{equation}\label{tensorcomplx2}
\xymatrix@=40pt{
0\ar@{~>}[r]&V_n\ar@{~>}[r]^{f_n}&V_{n+1}\ar@{~>}[r]^{f_{n+1}}&\cdots\ar@{~>}[r]^{f_{n+m-2}}&V_{n+m-1}\ar@{~>}[r]^{f_{n+m-1}}&V_{n+m}\ar@{~>}[r]&0,
}
\end{equation} 
where $n,m\in \Z$. For all $n\leq j\leq n+m$ , denote by $\widehat{\Fun}_{\A,V_j}$ the deformation functor corresponding to the finite dimensional $\A$-module $V_j$ as in \cite[\S 2]{blehervelez}. Then by Remark \ref{rem3.2} (ii), there is a well-defined injective map between set of deformations 
\begin{align}\label{deformmap}
\Xi_R: \widehat{\Fun}_{\Ar,\V}(R)&\to\prod_{j=n}^{n+m}\widehat{\Fun}_{\A,V_j}(R)\\
[\M,\widehat{\phi}]&\mapsto \{[M_j, \phi_j]\}_{j=n}^{n+m}\notag
\end{align}

\item Let $\k[\epsilon]$, with $\epsilon^2=0$, denote the the ring of dual numbers over $\k$. The {\it tangent space} of $\widehat{\Fun}_{\Ar, \V}$ and of $\Fun_{\Ar, \V}$ is defined to be the set $t_{\V}=\Fun_{\Ar, \V}(\k[\epsilon])$. By \cite[Lemma 2.10]{sch}, $t_{\V}$ has an structure of $\k$-vector space. Moreover, by using the injective map  (\ref{deformmap}) together with \cite[Prop. 2.1]{blehervelez}, we obtain that $\dim_\k t_{\V}<\infty$.
\end{enumerate}

\end{remark}

Following \cite[\S 14]{mazur}, we say that a functor $H: \widehat{\Ca}\to \Sets$ is {\it continuous} provided that for all objects $R$ in $\widehat{\Ca}$, 
\begin{equation}\label{cont}
H(R)=\invlim_iH(R/\mathfrak{m}_R^i),
\end{equation}
where $\mathfrak{m}_R$ denotes the unique maximal ideal of $R$. 

%The following result provides a version of \cite[Prop. 2.1]{blehervelez} for $\Ar$. 

%\begin{proposition}\label{prop3.5}
%The functor $\Fun_{\Ar,\V}$ has a pro-representable hull $R(\Ar,\V)\in \Ob(\widehat{\Ca})$ (in the sense of \cite[Def. 2.7]{sch}), and $\widehat{\Fun}_{\Ar,\V}$ is continuous in the sense of (\ref{cont}). Moreover, there exists a $\k$-vector space  isomorphism $t_{\V}\cong \Ext_{\Ar}^1(\V, \V)$. If $\SEnd_{\Ar}(\V)=\k$, then $\widehat{\Fun}_{\A.\V}$ is represented by $R(\A,\V)$. 
%\end{proposition}

%\subsection{Proof of Proposition \ref{prop3.5}}
We need to recall the following definition from \cite[Def. 1.2]{sch}.

\begin{definition}\label{defi3.7}
Let $\theta:R\to R_0$ be a morphism of Artinian objects in $\widehat{\Ca}$. We say that $\theta$ is a {\it small extension} if the kernel of $\theta$ is a non-zero principal ideal $tR$ that is annihilated by the unique maximal ideal $\m_R$ of $R$.
\end{definition}

In what follows, we adjust the ideas in \cite{blehervelez} in order to prove that $\Fun_{\Ar,\V}$ satisfies Schlessinger's criteria (H$_1$)-(H$_3$), and that if $\SEnd_{\Ar}(\V)=\k$, then $\Fun_{\Ar,\V}$ also satisfies (H$_4$) (see \cite[Thm. 2.11]{sch}).  More precisely, for all pullback diagrams of Artinian objects in $\widehat{\Ca}$, 

\begin{equation}\label{pullbackart}
\xymatrix@=20pt{
&R'''\ar[dl]_{\pi'}\ar[dr]^{\pi''}&\\
R'\ar[dr]_{\theta'}&&R''\ar[dl]^{\theta''}\\
&R&
}
\end{equation}
the induced map
\begin{equation}\label{thetapullback}
\Theta: \Fun_{\Ar, \V}(R''')\to \Fun_{\Ar, \V}(R') \times_{\Fun_{\Ar, \V}(R)}\Fun_{\Ar, \V}(R'')
\end{equation}
satisfies the following properties:
\begin{enumerate}
\item[(H$_1$)] $\Theta$ is surjective whenever $\theta'':R''\to R$ is a small extension as in Definition \ref{defi3.7};
\item[(H$_2$)] $\Theta$ is a bijective when $R=\k$ and $R''=\k[\epsilon]$, where $\k[\epsilon]$ denotes the ring of dual numbers over $\k$ as in Remark \ref{rem3.3} (ii);
\item[(H$_3$)] the $\k$-vector space $t_{\V}$ as in Remark \ref{rem3.3} (iii) is finite dimensional.
\end{enumerate}
If $\SEnd_{\Ar}(\V)=\k$, then the following property is also satisfied: 
\begin{enumerate}
\item[(H$_4$)] $\Theta$ is a bijection when $R'=R''$ and $\theta':R'\to R$ is a small extension. 
\end{enumerate}

\begin{remark}\label{rem3.4}
Note that by Remark \ref{rem3.3} (iii), we have already that $\dim _\k t_{\V}<\infty$, and thus Schlessinger's criterion (H$_3$) is satisfied by $\Fun_{\Ar,\V}$. 
\end{remark}
We first need to prove the following series of results, some of them previously obtained for finite dimensional modules over Frobenius $\k$-algebras in \cite{blehervelez}. It is important to recall that $\Ar$ is an infinite dimensional $\k$-algebra without identity. 

\begin{lemma}\label{lemma3.6}
Let $R$ be a fixed Artinian object in $\widehat{\Ca}$. 
\begin{enumerate} 
\item If $\P$ is a projective left (resp. right) $\Ar$-module with finite dimension over $\k$, then the $R\Ar$-module $\P_R=R\otimes_{\k, \iota_R}\P$ defines a projective $R\Ar$-module cover of $\P$, namely $\widehat{\eta}:\P_R\to \P$, where $\iota_R:\k \to R$ is the unique morphism of Artinian objects in $\widehat{\Ca}$ that endows $R$ with a $\k$-algebra structure. Moreover, $(\P_R, \pi_{R, \P})$ defines a lift of $\P$ over $R$, where $\pi_{R,\P}$ is the natural isomorphism of $\Ar$-modules $\k\otimes_R\P_R\to \P$. 
\item If $\Q$ is a indecomposable projective $R\Ar$-module that is free with finite rank over $R$, then there exists an orthogonal primitive idempotent $\epsilon^R$ in $R\A$ such that $\Q$ is of the form
\begin{equation*}
\xymatrix@=40pt{
\cdots\ar@{~>}[r]&0\ar@{~>}[r]&(R\A)\epsilon^R\ar@{~>}[r]^{f_R}&D_R(\epsilon^R(R\A))\ar@{~>}[r]&0\ar@{~>}[r]&\cdots.
}
\end{equation*} 
\item If $\Q$ is a projective left $R\Ar$-module that is free with finite rank over $R$, then $D_R(\Q)= \Hom_R(\Q, R)$ is also a projective right $R\Ar$-module which is also free with finite rank over $R$. 
\item Let $\theta: R\to R_0$ be a surjection of Artinian rings in $\widehat{\Ca}$, and let $\Q_0$ be a projective left (resp. right) $R_0\Ar$-module that is free with finite rank over $R_0$. Then there exists a projective left (resp. right) $R\Ar$-module $\Q$ with finite rank over $R$ and an isomorphism $\h: R_0\otimes_{R,\theta}\Q\to \Q_0$ of $R_0\Ar$-modules induced by $\theta$. 
\end{enumerate}
\end{lemma}

\begin{proof}
The statement (i) is straightforward. To prove (ii), note first that the natural projection $R\to \k$ induces a surjection of $\Q\to \k\otimes_R\Q$, where $\k\otimes_R\Q$ is an indecomposable projective $\Ar$-module. It follows from \cite[Prop. 6]{giraldo} that there exists an orthogonal primitive idempotent element $\epsilon$ in $\A$ such that $\P=\k\otimes_R\Q$ is as in (\ref{projindec}). It follows then that $\epsilon^R=1_R\otimes\epsilon$ is an orthogonal primitive idempotent in $R\A$ and thus the left $R\Ar$-module
\begin{equation}\label{projecR}
\P_R: \xymatrix@=40pt{
\cdots\ar@{~>}[r]&0\ar@{~>}[r]&(R\A)\epsilon^R\ar@{~>}[r]^{f_R}&D_R(\epsilon^R(R\A))\ar@{~>}[r]&0\ar@{~>}[r]&\cdots.
}
\end{equation} 
is a projective left $R\Ar$-module. The result now follows by Nakayama's Lemma and the fact that $\P_R$ induces a projective $R\Ar$-module cover of $\k\otimes_R\Q$ by (i).  To prove (iii), it is enough to assume that $\P$ is an indecomposable $\Ar$-module, and thus $\P$ is as in (\ref{projindec}). Thus $\P_R$ is a projective left $R\Ar$-module which has the form as in (\ref{projecR}). Note that $(R\A)\epsilon^R$ (resp. $\epsilon^R(R\A$)) is a projective left (resp. right) $R\A$-module. Then $D_R(\P_R)$ has the form 

\begin{equation}
\xymatrix@=40pt{
\cdots\ar@{~>}[r]&0\ar@{~>}[r]&\epsilon^R(R\A)\ar@{~>}[r]^{D_R(f_R)}&D_R((R\A)\epsilon^R)\ar@{~>}[r]&0\ar@{~>}[r]&\cdots,
}
\end{equation}

This implies that $D_R(\P_R)= R\otimes_{\k, \iota_R}D\P$, where $D\P$ is an indecomposable injective right $\Ar$-module. Since $\Ar$-mod is a Frobenius category, it follows that $D\P$ is also a projective right $\Ar$-module and thus by the dual of (i), we obtain that $D_R(\P_R)$ is a projective right $R\Ar$-module. Finally, to prove (iv), we can assume without loss of generality that $\theta$ is a small extension as in Definition \ref{defi3.7} and that $\Q_0$ is an indecomposable projective left $R_0\Ar$-module. Then by (iii), there exists an orthogonal primitive idempotent $\epsilon^{R_0}\in R_0\Ar$ such that $\Q_0$ is of the form 
\begin{equation}\label{projecR0}
\xymatrix@=40pt{
\cdots\ar@{~>}[r]&0\ar@{~>}[r]&(R_0\A)\epsilon^{R_0}\ar@{~>}[r]^{f_{R_0}}&D_{R_0}(\epsilon^{R_0}(R_0\A))\ar@{~>}[r]&0\ar@{~>}[r]&\cdots.
}
\end{equation} 
Since $\theta: R\to R_0$ is a small extension, there exists a short exact sequence of $R$-modules
\begin{equation}\label{sessext}
0\to tR\to R\xrightarrow{\theta} R_0\to 0. 
\end{equation}
After tensoring (\ref{sessext}) with $\A$ over $\k$, we obtain a short exact sequence of $R$-modules
\begin{equation*}
0\to tR\otimes_\k\A\to R\A\xrightarrow{\theta} R_0\A\to 0. 
\end{equation*}
This implies that $R\A/(tR\otimes_\k\A)\cong R_0\A$. On the other hand, we also have that $\m_R\otimes_\k\A\subseteq \mathrm{rad}\,R\A$.  Therefore the conditions of the Theorem on Lifting Idempotents in \cite[\S 6.7]{curtis} are satisfied.  Hence, there exists an orthogonal idempotent element $\epsilon^R\in R\A$ that induces $\epsilon^{R_0}$ under the natural projection $R\A\to R_0\A$. Let $\Q$ be the left $R\Ar$-module as in (\ref{projecR}). Thus $\Q$ is an indecomposable projective left $R\Ar$-module such that there exists an isomorphism $h:R_0\otimes_{R,\theta}\Q\to \Q_0$ induced by $\theta$. This finishes the proof of Lemma \ref{lemma3.6}.
\end{proof}

Note that Lemma \ref{lemma3.6} (iii) provides a version of \cite[Claim 7]{blehervelez}, where the definition of a Frobenius $\k$-algebra and \cite[Thm. 2.38]{curtis} are used. In our proof, we use instead the description of the indecomposable projective $\Ar$-modules as in (\ref{projindec}). 

The following result uses the same ideas as those in the proof of \cite[Claim 1]{blehervelez} and the fact that the projective $\Ar$-modules are also injective. 

\begin{lemma}\label{lemma3.7}
Let $\theta: R\to R_0$ be a surjection between Artinian objects in $\widehat{\Ca}$. Let $\M$, $\Q$ (resp. $\M_0$ , $\Q_0$) be left $R\Ar$-modules (resp. $R_0\Ar$-modules) that are free over $R$ (resp. $R_0$) with finite rank, and assume that $\Q$ and $\Q_0$ are projective. Suppose further that there are isomorphisms of left $R_0\Ar$-modules $ \g: R_0\otimes_{R,\theta}\M\to \M_0$ and $\h: R_0\otimes_{R, \theta}\Q\to \Q_0$. If $\widehat{\nu}_0\in \Hom_{R_0\Ar}(\M_0,\Q_0)$, then there exists $\widehat{\nu}\in \Hom_{R\Ar}(\M, \Q)$ such that $\widehat{\nu}_0=\h\circ(\mathrm{id}_{R_0}\otimes_{R,\theta}\widehat{\nu})\circ \g^{-1}$.
\end{lemma}

As in the proof of \cite[Claim 2]{blehervelez}, we use Lemma \ref{lemma3.6} (iv) together with Lemma \ref{lemma3.7} to obtain the following result.

\begin{lemma}\label{lemma3.8} 
Let $\theta$, $\M$, $\M_0$ and $\g$ be as in Lemma \ref{lemma3.7}. If $\widehat{\sigma}_0\in \End_{R_0\Ar}(\M_0)$ factors through a projective left $R_0\Ar$-module, then there exists $\widehat{\sigma}\in \End_{R\Ar}(\M)$ such that $\widehat{\sigma}$ factors through a projective left $R\Ar$-module and $\widehat{\sigma}_0=\g\circ (\mathrm{id}_{R_0}\otimes_{R,\theta}\widehat{\sigma})\circ \g^{-1}$.
\end{lemma}

%Note that (C2) is obtained by using (C1) together with Lemma \ref{lemma3.15}. The proof of Theorem \ref{thm3.16} (i) now follows by using the arguments in the proof of \cite[Claim 3]{blehervelez} together with (C2). Next assume that $\P$ is a projective left $\Ar$-module with finite dimension over $\k$. Let $(\P_R, \widehat{\pi}_{R, \P})$ be the lift of $\P$ over $R$ as in Lemma \ref{lemma3.14} (i). The following statement is proved by using (C1) together with the arguments in the proof of \cite[Claim 6]{blehervelez}. 

As a consequence of Lemma \ref{lemma3.8}, we obtain the following result, which can be proved by using the same ideas as those in the proof of \cite[Claim 1]{blehervelez}.

\begin{lemma}\label{lemma3.9}
Assume that $\SEnd_{\Ar}(\V)=\k$ and that $R$ is an Artinian ring in $\widehat{\Ca}$. If $(\M,\widehat{\phi})$ and $(\M',\widehat{\phi}')$ are lifts of $\V$ over $R$ such that there is an isomorphism of $R\Ar$-modules $\widehat{h}_0:\M\to \M'$, then there exists an isomorphism of $R\Ar$-modules $\widehat{h}: \M\to \M'$ such that $\widehat{\phi}'\circ (\mathrm{id}_\k\otimes\widehat{h}) =\widehat{\phi}$, i.e., $[\M, \widehat{\phi}]=[\M',\widehat{\phi'}]$ in $\Fun_{\Ar,\V}(R)$. 
\end{lemma}

The following result can be proved by using the ideas in that of \cite[Prop. 4.3]{bleher14}. However, we decided to include a proof for the convenience of the reader.  

\begin{lemma}\label{lemma3.10}
Let $\V$ and $R$ be as in Lemma \ref{lemma3.9}. If $(\M, \widehat{\phi})$ is a lift of $\V$ over $R$, then the ring homomorphism $\sigma_{\M}:R\to\SEnd_{R\Ar}(\M)$ coming from the action of $R$ on $\M$ via scalar multiplication is surjective.
\end{lemma}

\begin{proof}
Let $R_0$ be an Artinian object in $\widehat{\Ca}$ such that  $\theta: R\to R_0$ is a small extension as in Definition \ref{defi3.7} and that for all lifts $(\M_0,\widehat{\phi}_0)$ of $\V$ over $R_0$, the ring homomorphism $\sigma_{\M_0}: R_0\to \SEnd_{R_0\Ar}(\M_0)$ coming from the action of $R_0$ on $\M_0$ via scalar multiplication is surjective. Let $\widehat{f}\in \SEnd_{R\Ar}(\M)$. Then $\theta$ induces $\widehat{f}_0\in \SEnd_{R_0\Ar}(\M_0)$, where $\M_0=R_0\otimes_{R,\theta}\M$. By assumption, there exists $r_0\in R_0$ such that $\widehat{f}_0=\widehat{\mu}_{r_0}$, where $\widehat{\mu}_{r_0}$ denotes multiplication by $r_0$. This implies that there exists $\widehat{g}_0:\M_0\to \M_0$ that factors through a projective $R_0\Ar$-module such that $\widehat{f}_0-\widehat{\mu}_{r_0}=\widehat{g}_0$. By Lemma \ref{lemma3.8}, there exists $\widehat{g}: \M\to \M$ that factors through a projective $R\Ar$-module such that $\mathrm{id}_{R_0}\otimes \widehat{g}=\widehat{g}_0$. Let $r\in R$ such that $\theta(r)=r_0$, and let $\widehat{\mu}_r:\M\to \M$ be the scalar multiplication morphism by $r$.  Thus $\mathrm{id}_{R_0}\otimes \widehat{\mu}_r=\widehat{\mu}_{r_0}$. Let $\widehat{\beta}$ be the endomorphism of $\M$ given by  $\widehat{f}-\widehat{\mu}_r-\widehat{g}$. It follows that $\mathrm{id}_{R_0}\otimes \widehat{\beta}=0$. Since $\theta$ is a small extension, we have that there is $t\in R$ such that $\ker \theta = t R$. This implies that $\mathrm{im}\, \widehat{\beta}\subseteq t\M$ and thus $\widehat{\beta} \in \Hom_{R\Ar}(\M, t\M)$. Since $tR\cong \k$, it follows that $t\M\cong \k\otimes_R\M$ and there is an isomorphism $\Hom_{R\Ar}(\M,t\M)\cong\Hom_{\Ar}(\k\otimes_R\M,\k\otimes_R\M)\cong\End_{R\Ar}(\V)$. Let $\widehat{\beta}_0$ be the image of $\widehat{\beta}$ under this isomorphism. Then there exist $\lambda_0\in \k$ and $\widehat{\alpha}_0:\V\to\V$ that factors through a finite dimensional projective $\Ar$-module such that $\widehat{\beta}_0=\widehat{\mu}_{\lambda_0}+\widehat{\alpha}_0$, where $\widehat{\mu}_{\lambda_0}:\V\longrightarrow\V$ is multiplication by $\lambda_0$. Once again, by Lemma \ref{lemma3.8}, there exists $\widehat{\alpha}:\M\to\M$ such that $\widehat{\alpha}$ factors through a finitely generated projective $R\Ar$-module such that $\mathrm{id}_\k\otimes\widehat{\alpha}=\widehat{\alpha}_0$. Thus $\widehat{\beta}=\widehat{\mu}_{t\lambda}+\widehat{\alpha}$, where $\lambda\in R$ is such that $t\lambda$ gives $\lambda_0$ under the isomorphism $tR\cong \k$, and $\mu_{t\lambda}: \M\to t\M$ is multiplication by $t\lambda$. This implies that  $\widehat{f}= \widehat{\mu}_{t\lambda+r}+(\widehat{g}+\widehat{\alpha})$ in $\End_{R\Ar}(\M)$, and since $\widehat{\alpha}+\widehat{g}$ factors through a projective $R\Ar$-module, we get that $\widehat{f}=\widehat{\mu}_{t\lambda+r}$ in $\SEnd_{R\Ar}(\M)$.  Then $\widehat{f}=\widehat{\sigma}_{\M}(t\lambda+r)$ with $t\lambda+r\in R$ and thus $\sigma_{\M}$ is surjective. 
\end{proof}

The proof of the following result uses the same arguments in that of \cite[Claim 6]{blehervelez} and once again together with the fact that the projective $\Ar$-modules are also injective. 
\begin{lemma}\label{lemma3.11}
Let $R$ be an Artinian ring in $\widehat{\Ca}$. Suppose $\P$ is a finite dimensional projective $\Ar$-module and that there exists a commutative diagram of finite dimensional  $R\Ar$-modules 
\begin{equation}\label{diag1}
\begindc{\commdiag}[330]
\obj(0,1)[p0]{$0$}
\obj(2,1)[p1]{$\P_R$}
\obj(4,1)[p2]{$\widehat{T}$}
\obj(6,1)[p3]{$\widehat{C}$}
\obj(8,1)[p4]{$0$}
\obj(0,-1)[q0]{$0$}
\obj(2,-1)[q1]{$\P$}
\obj(4,-1)[q2]{$\k\otimes_R\widehat{T}$}
\obj(6,-1)[q3]{$\k\otimes_R\widehat{C}$}
\obj(8,-1)[q4]{$0$}
\mor{p0}{p1}{}
\mor{p1}{p2}{$\widehat{\alpha}_R$}
\mor{p2}{p3}{$\widehat{\beta}_R$}
\mor{p3}{p4}{}
\mor{q0}{q1}{}
\mor{q1}{q2}{$\widehat{\alpha}$}
\mor{q2}{q3}{$\widehat{\beta}$}
\mor{q3}{q4}{}
\mor{p1}{q1}{}
\mor{p2}{q2}{}
\mor{p3}{q3}{}
\enddc
\end{equation}
in which $\P_R$ is as in Lemma \ref{lemma3.6} (i), $\widehat{T}$ as well as  $\widehat{C}$ are free over $R$, and the bottom row arises by tensoring the top row with $\k$ over $R$ and by identifying $\P$ with $\k\otimes_R\P_R$. Then the top row of (\ref{diag1}) splits as a sequence of $R\Ar$-modules.
\end{lemma}

In the proof of the following result, we adjust the ideas in the proof of \cite[Claim 1, pg. 105]{blehervelez2} to our context and use the description of lifts of finite dimensional $\Ar$-modules over objects in $\widehat{\Ca}$ as in Remark \ref{rem3.2} (i). 

\begin{lemma}\label{lemma3.12}
Schlessinger's criterion \textup{(H$_1$)} is satisfied by $\Fun_{\Ar, \V}$.
\end{lemma}

\begin{proof}
Assume that $\theta''$ in (\ref{pullbackart}) is a small extension as in Definition \ref{defi3.7}. Let $(\M',\widehat{\phi}')$ and $(\M'', \widehat{\phi}'')$ be lifts of $\V$ over $R'$ and $R''$, respectively, and such that there exists an isomorphism of $R\A$-modules 
\begin{equation*}
\tau_R: R\otimes_{R',\theta'}\M'\to R\otimes_{R'',\theta''}\M''
\end{equation*}
that satisfies $\widehat{\phi}'_{\theta'}=\widehat{\phi}''_{\theta''}\circ (\mathrm{id}_\k\otimes \tau_R)$. Assume that $\M'=(M'_i, g'_i)_{i\in \Z}$ and $\M''= (M''_i,g''_i)_{i\in \Z}$. Define $\M''' = (M'''_i, g'''_i)_{i\in \Z}$ as follows. For all  $i\in \Z$, let $M'''_i= M'_i\times_{R\otimes_{R',\theta'}M'_i}M''_i$ and $g'''_i = (g'_i, g''_i)$. Since $\M'$ and $\M''$ are finitely generated as an $R'\Ar$-module and as an $R''\Ar$-module, respectively,  it follows that $\M'''$ is as finitely generated as an $R'''\Ar$-module. Since $\theta''$ is surjective, and $\M'$  and $\M''$ are free over $R'$ and $R''$, respectively, it follows by \cite[Lemma 3.4]{sch} that for each $i\in \Z$, $M'''_i$ is free over $R'''$, and thus so is $\M'''$. Moreover,  we also have that $R\otimes_{R'''}\M'''\cong R\otimes_{R',\theta'}\M'$ and thus we can define $\widehat{\phi}''':\k\otimes_{R'''}\M'''\to \V$ as the composition
\begin{equation*}
\k\otimes_{R'''}\M'''\cong \k\otimes_{R'}\M'\xrightarrow{\widehat{\phi}'}\V.
\end{equation*} 
Therefore, $(\M''', \widehat{\phi}''')$ defines a lift of $\V$ over $R'''$ such that 
\begin{align*}
\Fun_{\A,\V}(\pi')([\M''', \widehat{\phi}''']) = [\M',\widehat{\phi}'] \text{ and } \Fun_{\A,\V}(\pi'')([\M''', \widehat{\phi}''']) = [\M'',\widehat{\phi}''].
\end{align*}
Thus $\Theta$  in (\ref{thetapullback}) is surjective, which proves that (H$_1$) is satisfied by $\Fun_{\Ar, \V}$.  
\end{proof}

The proof of the following result is obtained by adapting the arguments in the proof of \cite[Claim 5]{blehervelez} to our situation and by replacing Claim 2, Claim 3 and Claim 4 in \cite{blehervelez} by Lemmata \ref{lemma3.8}, \ref{lemma3.9} and \ref{lemma3.10}, respectively.

\begin{lemma}\label{lemma3.13}
Consider the pullback diagram of Artinian rings in $\widehat{\Ca}$ and let $\Theta$ be as in (\ref{thetapullback}). 
\begin{enumerate}
\item If $R=\k$, then $\Theta$ is injective. In particular, Schlessinger's criterion \textup{(H$_2$)} is always satisfied by $\Fun_{\Ar,\V}$.  
\item If $\SEnd_{\Ar}(\V)=\k$, then $\Theta$ is injective. In particular, Schlessinger's criterion \textup{(H$_4$)} is satisfied by $\Fun_{\Ar,\V}$ in this situation. 
\end{enumerate}
\end{lemma}

Although we know already that $\Fun_{\Ar,\V}$ satisfies Schlessinger's criterion \textup{(H$_3$)}, we still need the following result concerning the tangent space of $\Fun_{\Ar,\V}$, which can be proved by using the same ideas as in \cite[Prop. 2.1]{blehervelez}. However, we decided to include a proof for the convenience of the reader.

\begin{lemma}\label{lemma3.14}
There exists an isomorphism of $\k$-vector spaces $t_{\V}\cong \Ext_{\Ar}^1(\V,\V)$. 
\end{lemma}

%Since we are restricting ourselves to finite dimensional $\Ar$-modules, the arguments needed to obtain Lemma \ref{lemma3.14} are exactly the same as those within that of \cite[Prop. 2.1]{blehervelez}. However, we decided to include a proof for the convenience of the reader. 
\begin{proof}%[Proof of Lemma \ref{lemma3.14}]
As in the proof of \cite[Prop. 2.1]{blehervelez}, the isomorphism $t_{\V}\cong \Ext_{\Ar}^1(\V,\V)$ is established in a similar manner to that in \cite[\S 22]{mazur}. Namely, for a given lift $(\M, \widehat{\phi})$ of $\V$ over the ring of dual numbers $\k[\epsilon]$, we have $\Ar$-module isomorphisms $\epsilon \M\cong\V$ and $\M/\epsilon \M\cong \V$. Therefore, we obtain a short exact sequence of $\Ar$-modules $\mathcal{E}_{\M}: 0\to \V\to \M\to \V\to 0$.
Thus we have a well defined $\k$-linear map $s:t_{\V}\to \Ext_{\Ar}^1(\V,\V)$ defined as $s([\M,\widehat{\phi}])= \mathcal{E}_{\M}$. Now, if $\mathcal{E}: 0\to \V\xrightarrow{\widehat{\alpha}_1}\V_1\xrightarrow{\widehat{\alpha}_2}\V\to 0$
is an extension of $\Ar$-modules, then $\V_1$ is naturally a $\k[\epsilon]\Ar$-module by letting $\epsilon\cdot v_1 = (\widehat{\alpha_1}\circ \widehat{\alpha}_2)(v_1)$ for all $v_1\in \V_1$.   In this way, $\V_1$ is free over $\k[\epsilon]$ and there is a $\Ar$-module isomorphism $\widehat{\psi}: \V_1/\epsilon \V_1\to \V$. Therefore, $(\V_1,\widehat{\psi})$ defines a lift of $\V$ over $\k[\epsilon]$. Thus the $\k$-linear map $s': \Ext_{\Ar}^1(\V,\V)\to t_{\V}$ defined as $s'(\mathcal{E})= [\V_1, \widehat{\psi}]$ gives an inverse of $s$, which implies that $t_{\V}\cong \Ext_{\Ar}^1(\V,\V)$ as $\k$-vector spaces. 
\end{proof}

%Note that In the proof of Lemma \ref{lemma3.14}, we use the fact that $\V$ is a finite dimensional $\Ar$-module in order to obtain the $\Ar_1$-module in $\V_1$ in (\ref{extV1}) is finitely generated over $\k[\epsilon]$.    

Note that by \cite[Prop. 2.1]{blehervelez}, $t_V\cong \Ext_{\A}^1(V,V)$ for all finitely generated $\A$-modules $V$, in particular for when $V=\A$. However,  Lemma \ref{lemma3.14} cannot be applied when $\V= \Ar$, for $\Ar$ has infinite $\k$-dimension.

We next prove the continuity of the deformation functor $\widehat{\Fun}_{\Ar,\V}$ in the next Lemma \ref{lemma3.15}. As with the proof of Lemma \ref{lemma3.14}, the proof of Lemma \ref{lemma3.15} can be obtained by adapting the arguments in the proofs of \cite[Prop. 2.1]{blehervelez} and \cite[Prop. 2.4.4]{blehervelez2} to our context (see also \cite[\S 20, Prop. 1]{mazur}). We decided to include a proof for the convenience of the reader. 

\begin{lemma}\label{lemma3.15}
The functor $\widehat{\Fun}_{\A,\V}:\widehat{\Ca}\to \Sets$ is continuous.
\end{lemma}

%As noted before the proof of Lemma \ref{lemma3.14}, since we are considering only finite dimensional $\Ar$-modules, the proof of Lemma \ref{lemma3.15} can be obtained by using the ideas within the proof of \cite[Prop. 2.1]{blehervelez}. However, for the convenience of the reader, we decided to include a proof including some details omitted in that of   \cite[Prop. 2.1]{blehervelez}.

\begin{proof}
For all objects $R$ in $\widehat{\Ca}$, we consider the natural map 
\begin{equation}\label{contfunt}
\Gamma: \widehat{\Fun}_{\Ar,\V}(R)\to \invlim_i\Fun_{\Ar,\V}(R_i)
\end{equation}
defined by $\Gamma([\M, \widehat{\phi}])=\{[\M_i,\widehat{\phi}_i]\}_{i=1}^\infty$, where for all $i\geq 1$, $R_i=R/\m_R^i$, $\pi_i:R\to R_i$ is the natural projection, and $[\M_i,\widehat{\phi}_i]=\widehat{\Fun}_{\A,\V}(\pi_i)([\M,\widehat{\phi}])$. We first prove that $\Gamma$ as in (\ref{contfunt}) is surjective. Let $\{[\N_i,\widehat{\psi}_i]\}_{i=1}^\infty\in  \invlim_i\Fun_{\Ar,\V}(R_i)$. Then for each $i\geq 1$, there exists an isomorphism of $R_i\Ar$-modules
\begin{equation*}
\widehat{\alpha}_i:R_i\otimes_{R_{i+1}}\N_{i+1}\to \N_i
\end{equation*}
such that $\widehat{\psi}_i\circ (\mathrm{id}_\k\otimes \widehat{\alpha}_i)=\widehat{\psi}_{i+1}$. Let $(\N'_1,\widehat{\psi}'_1)= (\N_1,\widehat{\psi}_1)$. For each $i\geq 2$, the natural surjection $R_{i+1}\to R_i$ induces a surjective $R_{i+1}\Ar$-module homomorphism
\begin{equation*}
\N_{i+1}\to R_i\otimes_{R_{i+1}}\N_{i+1}.
\end{equation*}
\noindent
Hence, for all $i\geq 2$, define $(\N'_i,\widehat{\psi}'_i)$ as follows: $\N'_i = R_i\otimes_{R_i+1}\N_{i+1}$ and $\widehat{\psi}'_i$ is the composition
\begin{equation*}
\k\otimes_{R_i}\N'_i\cong \k\otimes_{R_{i+1}}\N_{i+1}\xrightarrow{\widehat{\psi}_{i+1}}\V.
\end{equation*}
Then for all $i\geq 1$, $(\N'_i,\widehat{\psi}'_i)$ is a lift $\V$ over $R_i$, which is  also isomorphic to $(\N_i, \widehat{\psi}_i)$  and there exists surjective $R_i\Ar$-module homomorphisms 
\begin{equation*}
\widehat{\beta}^{i+1}_i: \N'_{i+1}\to \N'_i 
\end{equation*}
that induces a natural isomorphism $R_i\otimes_{R_{i+1}}\N'_{i+1} = \N'_i$ which preserves the $\Ar$-module structure. Moreover $\widehat{\psi}'_i$ is equal to the composition 
\begin{equation*}
\k\otimes_{R_i}\N'_i=\k\otimes_{R_i}(R_i\otimes_{R_{i+1}}\N'_{i+1})\cong \k\otimes_{R_{i+1}}\N'_{i+1}\xrightarrow{\widehat{\psi}'_{i+1}}\V.  
\end{equation*}
Therefore $\{[\N'_i, \widehat{\psi}_i],\widehat{\beta}^{i+1}_i\}_{i=1}^\infty$ forms an inverse system of deformations of $\V$. If we let $\N'=\invlim_i\N'_i$  and $\widehat{\psi}=\invlim_i\widehat{\psi}'_i$, then we have that $(\N', \widehat{\psi}')$ is a lift of $\V$ over $R=\invlim_iR_i$ such that $\Gamma([\N', \widehat{\psi}'])=\{[\N_i, \widehat{\psi}_i]\}_{i=1}^\infty$. This proves that $\Gamma$ is surjective. In order to prove that $\Gamma$ is injective, assume that $[\M, \widehat{\phi}]$ and $[\M', \widehat{\phi}']$ are lifts of $\V$ over $R$ such that $\Gamma([\M, \widehat{\phi}])=\Gamma([\M', \widehat{\phi}'])$. Then for all $i\geq 1$, there is an isomorphism of $R_i\Ar$-modules 
\begin{equation*}
\widehat{\gamma}_i:R_i\otimes_R\M\to R_i\otimes_R\M', 
\end{equation*}
such that $\widehat{\phi}'\circ (\mathrm{id}_\k\otimes_{R_i}\widehat{\gamma}_i)=\widehat{\phi}$. 
For each $i\geq 1$, define 
\begin{equation*}
\widehat{\zeta}_i=\widehat{\gamma}_i^{-1}\circ (\mathrm{id}_{R_i}\otimes_{R_{i+1}}\widehat{\gamma}_{i+1})-\mathrm{id}_{R_i\otimes_R\M},
\end{equation*}
and set $\widehat{\zeta}_i^{(i)}=\widehat{\zeta}_i$. Observe that $\mathrm{id}_{\k}\otimes_{R_i}\widehat{\zeta}_i=0$. Since the $R_i$ are Artinian, for each $i\geq 2$, and $i\leq \ell< j$, we can lift the $R_\ell\Ar$-module homomorphism $\widehat{\zeta}_i^{(\ell)}\in \End_{R_\ell\Ar}(R_\ell\otimes_R \M)$ to an $R_j\Ar$-module homomorphism $\widehat{\zeta}_i^{(j)}\in \End_{R_j\Ar}(R_j\otimes_R \M)$ such that $\mathrm{id}_{R_\ell}\otimes_{R_j}\widehat{\zeta}_i^{(j)}=\widehat{\zeta}_i^{(\ell)}$. Moreover, since for all $2\leq i\leq \ell< j$, the morphism $\widehat{\zeta}_i^{(j)}$ is nilpotent, it follows that $\mathrm{id}_{R_j\otimes_R\M} + \widehat{\zeta}_{i}^{(j)}$ is invertible.  Next let $\widehat{f}_1 = \widehat{\gamma}_1$ and $\widehat{f}_2 = \widehat{\gamma}_2$, and for all $j\geq 3$, let 
\begin{align*}
\widehat{f}_j=\widehat{\gamma}_j\circ (\mathrm{id}_{R_j\otimes_R\M} + \widehat{\zeta}_{j-1}^{(j)})^{-1}\circ (\mathrm{id}_{R_j\otimes_R\M} + \widehat{\zeta}_{j-2}^{(j)})^{-1}\circ \cdots \circ (\mathrm{id}_{R_j\otimes_R\M} + \widehat{\zeta}_{2}^{(j)})^{-1}
\end{align*}
Note that $\mathrm{id}_{R_1}\otimes_{R_2}\widehat{f}_2=\widehat{f}_1$. On the other hand, if $j\geq 2$, then 
\begin{align*}
\mathrm{id}_{R_j}\otimes_{R_{j+1}}\widehat{f}_{j+1}&= (\mathrm{id}_{R_j}\otimes_{R_{j+1}}\widehat{\gamma}_{j+1})\circ  (\mathrm{id}_{R_j\otimes_R\M} + \widehat{\zeta}_{j}^{(j)})^{-1}\circ (\mathrm{id}_{R_j\otimes_R\M} + \widehat{\zeta}_{j-1}^{(j)})^{-1}\circ \cdots \circ (\mathrm{id}_{R_j\otimes_R\M} + \widehat{\zeta}_{2}^{(j)})^{-1}\\
&= \widehat{\gamma}_j\circ (\mathrm{id}_{R_j\otimes_R\M}+\widehat{\zeta}_j)\circ  (\mathrm{id}_{R_j\otimes_R\M} + \widehat{\zeta}_j)^{-1}\circ (\mathrm{id}_{R_j\otimes_R\M} + \widehat{\zeta}_{j-1}^{(j)})^{-1}\circ \cdots \circ (\mathrm{id}_{R_j\otimes_R\M} + \widehat{\zeta}_{2}^{(j)})^{-1}\\
&=\widehat{f}_j.
\end{align*}
Moreover, for each $i\geq 1$, we have 
\begin{align*}
\widehat{\phi}'\circ (\mathrm{id}_\k\otimes_{R_i} \widehat{f}_i)&= \widehat{\phi}'\circ (\mathrm{id}_\k\otimes_{R_i} \widehat{\gamma}_i)\circ (\mathrm{id}_{\k\otimes_R \M}+\mathrm{id}_\k\otimes_{R_{i-1}}\widehat{\zeta}_{i-1})^{-1}\circ \cdots \circ (\mathrm{id}_{\k\otimes_R \M}+\mathrm{id}_\k\otimes_{R_2}\widehat{\zeta}_{2})^{-1}\\
&= \widehat{\phi}. 
\end{align*}
Let $\widehat{f}=\invlim_i\widehat{f}_i$. Then $\widehat{f}:\M\to \M'$ is an isomorphism of $R\Ar$-modules such that $\widehat{\phi}'\circ (\mathrm{id}_\k\otimes_R\widehat{f})= \widehat{\phi}$. This proves that $\Gamma$ is injective. This finishes the proof of Lemma \ref{lemma3.15}.  
\end{proof}

As a consequence of Remark \ref{rem3.4} together with Lemmata \ref{lemma3.12}, \ref{lemma3.13}, \ref{lemma3.14}, and \ref{lemma3.15}, we obtain the following result, which is a version of \cite[Prop. 2.1]{blehervelez} for finite dimensional modules over $\Ar$.

\begin{proposition}\label{prop3.17}
The functor $\Fun_{\Ar,\V}$ has a pro-representable hull $R(\Ar,\V)$ in $\widehat{\Ca}$ as defined in \cite[Def. 2.7]{sch}, and the functor $\widehat{\Fun}_{\Ar,\V}$ is continuous. Moreover, there is an isomorphism of $\k$-vector spaces 

\begin{equation}\label{tangentspaces}
t_{\V}\to \Ext_{\Ar}^1(\V,\V).
\end{equation} 

If $\SEnd_{\Ar}(\V)= \k$, then $R(\Ar,\V)$ represents $\widehat{\Fun}_{\Ar,\V}$. 
\end{proposition}

The following definition is the version of \cite[Def. 2.3]{blehervelez} for finite dimensional modules over $\Ar$.  

\begin{definition}\label{def3.18}
By using the notation and the results in Proposition \ref{prop3.17}, it follows that there exists a deformation $[U(\Ar,\V), \phi_{U(\Ar,\V)}]$ of $\V$ over $R(\Ar,\V)$ such that for each $R$ in $\widehat{\Ca}$, the map 
\begin{equation}\label{repres}
\Hom_{\widehat{\Ca}}(R(\Ar,\V),R)\to \widehat{\Fun}_{\Ar,\V}(R)
\end{equation}
which sends $\theta\in \Hom_{\widehat{\Ca}}(R(\Ar,\V),R)$ to $ \widehat{\Fun}_{\Ar,\V}(\theta)([U(\Ar,\V), \phi_{U(\Ar,\V)}])$ is surjective, and this map is bijective provided that $R = \k[\epsilon]$ is the ring of dual numbers with $\epsilon^2=0$. Moreover, $R(\Ar, \V)$ represents $\widehat{\Fun}_{\A,\V}$ provided that (\ref{repres}) is bijective for all $R$ in $\widehat{\Ca}$. The ring $R(\Ar,\V)$ is called the {\it versal deformation ring} of $\V$, and $[U(\Ar,\V), \phi_{U(\Ar,\V)}]$ the {\it versal deformation} of $\V$. In general, $R(\Ar,\V)$ is unique up to (a non-canonical) isomorphism. 
If $R(\Ar,\V)$ represents $\widehat{\Fun}_{\Ar,\V}$, then in this situation $R(\Ar,\V)$ is called the {\it universal deformation ring} of $\V$, and $[U(\Ar,\V), \phi_{U(\Ar,\V)}]$ is the {\it universal deformation ring} of $\V$. In general, $R(\Ar,\V)$ is unique up to a canonical isomorphism. \end{definition}

\begin{remark}\label{remark3}
\begin{enumerate}
\item It follows from the isomorphism of $\k$-vector spaces (\ref{tangentspaces}) that if $\dim_\k \Ext_{\Ar}^1(\V,\V)=r$, then the versal deformation ring $R(\Ar,\V)$ is isomorphic to a quotient algebra of the power series ring  $\k[\![t_1,\ldots,t_r]\!]$ and $r$ is minimal with respect to this property. In particular, if $\V$ is such that $\Ext_{\Ar}^1(\V,\V)=0$, then $R(\Ar,\V)$ is universal and isomorphic to $\k$ (see \cite[Remark 2.1]{bleher15} for more details).

\item Because of the continuity of the deformation functor, most of the arguments concerning $\widehat{\Fun}_{\Ar,\V}$ can be carried out for $\Fun_{\Ar,\V}$, and thus we are able to restrict ourselves to discuss liftings of $\Ar$-modules over Artinian objects in $\widehat{\Ca}$. 
\end{enumerate}
\end{remark}

In the following, we assume that $\SEnd_{\Ar}(\V)=\k$. Note that by Proposition \ref{prop3.17}, the versal deformation ring $R(\Ar,\V)$ is universal. 

\begin{lemma}\label{lemma3.19}
Let $\P$ be a finite dimensional projective $\Ar$-module. Then the versal deformation ring $R(\Ar, \V\oplus \P)$ is universal and isomorphic to $R(\Ar, \V)$.
\end{lemma}

The proof of Lemma \ref{lemma3.19} is obtained by using  the same arguments as those in the proof of \cite[Thm. 2.6 (iii)]{blehervelez}. Namely, since $\SEnd_{\Ar}(\V\oplus \P)=\SEnd_{\Ar}(\V)=\k$, it follows by Proposition \ref{prop3.17} that $R(\Ar, \V\oplus \P)$ is universal. Moreover, we use Lemma \ref{lemma3.11} together with adapting the arguments in the proof of \cite[Prop. 2.6 \& Cor. 2.7]{bleher14} to our context. 

\begin{lemma}\label{lemma3.20}
The versal deformation ring $R(\Ar, \Omega_{\Ar} \V)$ is universal and isomorphic to $R(\Ar, \V)$.
\end{lemma}

The proof of Lemma \ref{lemma3.20} is obtained in the same way as that of \cite[Thm.2.6 (iv)]{blehervelez} by replacing Claims 1 and 7 in \cite{blehervelez} by Lemmata \ref{lemma3.7} and \ref{lemma3.6} (iii), respectively.

\begin{lemma}\label{lemma3.21}{(cf. \cite[Prop. 2.5]{gaviria-velez})}
Let $\nu_{\Ar}$ be as in Remark \ref{rem1.1} (v). Then the versal deformation ring $R(\Ar, \nu_{\Ar} \V)$ is universal and isomorphic to $R(\Ar, \V)$.
\end{lemma}

\begin{proof}
Note that since $\nu_{\Ar}$ is an automorphism of $\Ar$-\underline{mod}, then $\SEnd_{\Ar}(\nu_{\Ar}\V)\cong \SEnd_{\Ar}(\V)=\k$. Thus by Proposition \ref{prop3.17}, the versal deformation ring $R(\Ar, \nu_{\Ar}\V)$ is universal. It is straighforward to prove that for all Artinian $R$ in $\widehat{\Ca}$, there exists a bijection between set of deformations 
\begin{equation}
g_R:\Def_{\Ar}(\V,R)\to \Def_{\Ar}(\nu_{\Ar}\V,R)
\end{equation}
which is also natural with respect of morphisms of Artinian objects in $\widehat{\Ca}$. Since both $\widehat{\Fun}_{\Ar,\V}$ and $\widehat{\Fun}_{\Ar,\nu_{\Ar}\V}$ are both continuous by Proposition \ref{prop3.17}, it follows that $R(\Ar,\V)$ and $R(\Ar, \nu_{\Ar}\V)$ are isomorphic in $\widehat{\Ca}$.
\end{proof}

\subsection{Proof of Theorem \ref{thm1}}
In the following we end this section by using the previous results in order to prove Theorem \ref{thm1}. 
The statement in Theorem \ref{thm1} (i) follows from Proposition \ref{prop3.17} and Lemma \ref{lemma3.19}; Theorem \ref{thm1} (ii) follows from Lemma \ref{lemma3.20}; and Theorem \ref{thm1} (iii) follows from Lemma \ref{lemma3.20} together with Lemma \ref{lemma3.21} and Remark \ref{rem1.1} (v). 

\section{A particular example: the repetitive algebra of the Kronecker algebra}\label{sec4}
Let $\A_1=\k Q$ and $\Ar_1=\k\widehat{Q}/\langle\widehat{\rho}\rangle$ be the $2$-Kronecker $\k$-algebra and its repetitive algebra as in (\ref{kronecker}) and (\ref{repquiver}), respectively.  
%If $\Gamma_S(\Ar_1)$ the stable Auslander-Reiten quiver of $\A_1$, then it follows by \cite[\S II.7.3]{erdmann} that $\Gamma_S(\A_1)$ has infinitely many $1$-tubes and one component of tree type $\widetilde{\mathbb{A}}_{12}$.

%For all $z\in \Z$, if $v_z$ is a vertex of $\widehat{Q}$, then we denote $\1_{v_z}$ the corresponding string representative for $\Ar_1$ of length $0$, and by $\M[\1_{v_z}]$ the corresponding simple $\Ar_1$-module. 

Observe that for each $z\in \Z$, the radical series of the indecomposable projective $\Ar_1$-modules corresponding to the vertices $1_z$ and $2_z$ are given as in Figure \ref{fig1}, where for each vertex $v\in \Q_0$, $S_v$ denotes the corresponding simple $\Ar_1$-module. 
\begin{figure}[ht]
\begin{align*}
\xymatrix@=20pt{
&&S_{1_z}\ar[dl]_{\beta_z}\ar[dr]^{\alpha_z}&\\
\P_{1_z}:&S_{2_z}\ar[dr]_{\beta^\ast_z}&&S_{2_z}\ar[dl]^{\alpha^\ast_z}\\
&&S_{1_{z-1}}&
}&&&
\xymatrix@=20pt{
&&S_{2_z}\ar[dl]_{\beta^\ast_z}\ar[dr]^{\alpha^\ast_z}&\\
\P_{2_z}:&S_{1_{z-1}}\ar[dr]_{\beta_{z-1}}&&S_{1_{z-1}}\ar[dl]^{\alpha_{z-1}}\\
&&S_{2_{z-1}}&
}
\end{align*}
\caption{The radical series of the indecomposable projective $\Ar_1$-modules.}\label{fig1}.
\end{figure}
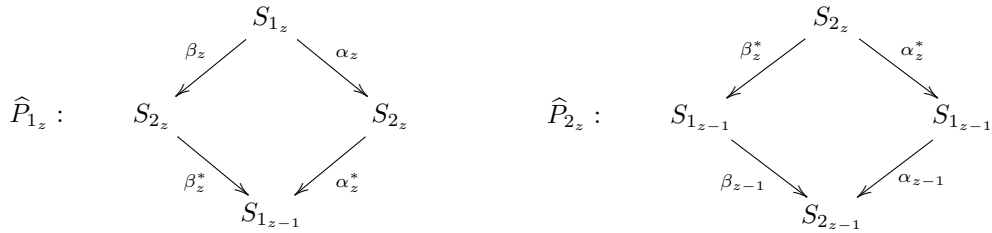

\begin{remark}\label{rem4.1}
As explained in e.g. \cite[\S II.1.3]{erdmann}, in order to study indecomposable non-projective $\Ar_1$-modules, irreducible morphisms between non-projective indecomposable modules over $\Ar_1$, and the stable Auslander-Reiten quiver of $\Ar_1$, it is enough to look at the string algebra (as in e.g. \cite[\S II.1.2]{erdmann}) $\Ar_0=\Ar_1/\mathrm{soc}\,\Ar_1$. In particular, this argument implies that all the indecomposable non-projective $\Ar_1$-modules are inflated from indecomposable $\Ar_0$-modules. In this article we are only interested in the class of indecomposable non-projective $\Ar_1$-modules that are inflations of so-called string $\Ar_0$-modules. It is important to mention that this approach was previously used in e.g. \cite[App.]{bleher9}, \cite[App. A]{calderon-giraldo-rueda-velez} and \cite[App. B]{blehervelez} in order to describe string modules over a number of symmetric special biserial algebras. 
\end{remark}

In the following, we review the definition of string modules over $\Ar_1$. Let $z\in \Z$ be fixed. For each arrow $\alpha_z, \alpha_z^\ast, \beta_z,\beta_z^\ast$ of $\Q$, we define a formal inverse by $\alpha_z^{-1}, (\alpha_z^\ast)^{-1}, \beta_z^{-1},(\beta_z^\ast)^{-1}$, respectively, and we let $\mathbf{s}(\alpha_z)=\mathbf{s}(\beta_z)=\mathbf{t}(\alpha_{z+1}^\ast)=\mathbf{t}(\beta_{z+1}^\ast)=\mathbf{t}(\alpha_z^{-1})=\mathbf{t}(\beta_z^{-1})=1_z$ and  $\mathbf{t}(\alpha_z)=\mathbf{t}(\beta_z)=\mathbf{s}(\alpha_z^\ast)=\mathbf{s}(\beta_z^\ast)=\mathbf{s}(\alpha_z^{-1})=\mathbf{s}(\beta_z^{-1})=2_z$, i.e. $\mathbf{s}$ and $\mathbf{t}$ indicate the vertex where an arrow or the formal inverse of an arrow starts and ends, respectively. By a {\it word} of length $n\geq 1$ we mean a sequence $w_n\cdots w_1$, where the $w_j$ is either an arrow or a formal inverse of an arrow, 
and where $\mathbf{s}(w_{j+1})=\mathbf{t}(w_{j})$ for  $1\leq j \leq n-1$. We define $(w_n\cdots w_1)^{-1}={w_1}^{-1}\cdots {w_n}^{-1}$, $\mathbf{s}(w_n\cdots w_1)=\mathbf{s}
(w_1)$ and $\mathbf{t}(w_n\cdots w_1)=\mathbf{t}(w_n)$.  If $v$ is a vertex of $\Q$, we define an empty word $\1_v$ of length zero with $\mathbf{t}(\1_v)=v=\mathbf{s}
(\1_v)$ and $(\1_v)^{-1}=\1_v$. 
We denote by $\mathcal{W}$ the set of all words and let 

\begin{equation*}
J=\{\beta_z^\ast\alpha_z,\alpha_{z-1}\beta_z^\ast,\alpha_z^\ast\beta_z, \beta_{z-1}\alpha_z^\ast,\alpha^\ast_z\alpha_z,\beta^\ast_z\beta_z, \alpha_{z-1}\alpha^\ast_z,\beta_{z-1}\beta^\ast_z: z\in \Z\}.
\end{equation*}
 
Note that $\Ar_0= \k \Q/\langle J \rangle$. Let $\sim$ be the equivalence relation on $\mathcal{W}$ defined by  $w\sim w'$ if and only if $w=w'$ or $w^{-1}
=w'$. A \textit{string} is a representative $C$ of an equivalence class under the relation $\sim$,  where either  $C=\1_v$ for some vertex $v$ of $Q$, or $C=w_n\cdots w_1$ with $n
\geq 1$ and $w_j\not=w_{j+1}^{-1}$ for $1\leq j\leq  n-1$ and no sub-word of $C$ or its formal inverse belong to $J$. If $C$ is a string such that $\mathbf{s}(C)=\mathbf{t}(C)$, then we let 
$C^0=\1_{\mathbf{t}(C)}$. If $C=w_n\cdots w_1$ and $D=v_m\cdots v_1$ are strings of 
length $n,m\geq 1$, respectively, we say that the composition $CD$ of $C$ and $D$ is defined provided that $w_n\cdots w_1v_m\cdots v_1$ is a string and write $CD=w_n\cdots 
w_1v_m\cdots v_1$. Observe that  $C\1_{\mathbf{s}(C)}\sim C$ and $\1_{\mathbf{t}(C)}C\sim C$. Moreover, if $C=w_n\cdots w_1$ is a string of length $n\geq 1$, then $C\sim w_n\cdots w_{j+1}\1_{\mathbf{t}(w_j)}w_j\cdots w_1$ for 
all $1\leq j\leq n-1$. If $C=w_n\cdots w_1$ is a string of length $n\geq 1$, then there exists an indecomposable non-projective $\Ar_1$-module $\M[C]$, called the {\it string 
module} corresponding to the string representative $C$, which can be  described as follows. There is an ordered $\k$-basis $\{z_0,z_1,\ldots,z_n\}$ of $\M[C]$ such that the action of $
\Ar_1$ 
on $\M[C]$ is given by the following representation $\varphi_C:\Ar_1\to \Mat(n+1,\k)$. Let $\text{\bf v}(j)=\mathbf{t}(w_{j+1})$ for $0\leq j\leq n-1$ and $\text{\bf v}(n)=\mathbf{s}(w_n)$. Then for each vertex $v\in \Q_0$, for each arrow $\gamma\in \Q_1$, and for all $0\leq j\leq n$, define
\begin{align*}
\varphi_C(v)(z_j)=\begin{cases}
z_j, &\text{ if $\text{\bf v}(j)=v$,}\\
0, &\text{ otherwise,}
\end{cases}
&&
\text{ and }
&&
\varphi_C(\gamma)(z_j)=\begin{cases}
z_{j-1}, &\text{ if $w_j=\gamma$,}\\
z_{j+1}, &\text{ if $w_{j+1}=\gamma^{-1}$,}\\
0, &\text{ otherwise.}
\end{cases}
\end{align*}

We call $\varphi_C$ the {\it canonical representation} and $\{z_0,z_1,\ldots,z_n\}$ a {\it canonical $\k$-basis} for $\M[C]$ relative to the string representative $C$. Note that $\M[C]
\cong \M[C^{-1}]$.  If $C=\1_v$ for some $v\in\Q_0$, 
then $\M[C]=S_v$ is the simple $\Ar_1$-module corresponding to the vertex $v$. From now on, we use $\M[\1_v]$ to denote the simple string $\Ar_1$-module corresponding to a vertex $v\in \Q_0$. It follows by Remark \ref{rem4.1} and \cite{buri} that the string $\Ar_1$-modules provide a description of indecomposable non-projective objects in $\Ar_1$-mod. 

The following definition is based on the description of irreducible morphisms between string modules for special biserial algebras as in e.g. \cite[\S II.5]{erdmann}. 

\begin{definition}\label{defi4.2}
Let $C$ and $C'$ be string representatives for $\Ar_1$.
\begin{enumerate}
\item $C$ is a {\it left hook} (resp. {\it right hook}) of $C'$ and write $C={_h}C'$ (resp. $C= C'_h$) provided that there exist arrows $\alpha, \beta \in \Q_1$ such that  $C$ is string equivalent to $\alpha\beta^{-1}C'$ (resp. $C'\alpha\beta^{-1}$);
\item $C$ is a {\it left co-hook} (resp. {\it right co-hook}) of $C'$ and write $C={_c}C'$ (resp. $C=C'_{c}$) provided that there exist arrows $\gamma, \delta \in \Q_1$ such that $C$ is string equivalent to $\gamma^{-1}\delta C'$ (resp. $C'\gamma^{-1}\delta$);
\item $C$ is said to be of {\it minimal string length} if there is no a string representative $C''$ for $\Ar_1$ such that $C={_h}C''$ or $C=C''_h$  or $C={_c}C''$ or $C=C''_c$.  
\end{enumerate}
It follows from \cite{buri} (see also \cite[\S II.5.3 \& \S II.6.3]{erdmann}) that all irreducible morphisms between string $\Ar_1$-modules are either canonical injections $\M[C]\to \M[C_h]$, $\M[C]\to \M[{_h}C]$, or canonical surjections $\M[C_c]\to \M[C]$, $\M[{_c}C]\to \M[C]$. 

\end{definition}

%For each $z\in \mathbb{Z}$, there exist two bands representatives for $\Ar_1$, namely $B_z=\alpha_z\beta_z^{-1}$ and $B_z^\ast =\alpha_z^\ast(\beta_z^\ast)^{-1}$.  

\begin{remark}
As in the statement of Theorem \ref{thm2}, we denote by $\Gamma_S(\Ar_1)$ the stable Auslander-Reiten quiver of $\Ar_1$, and let $\mathfrak{C}$ be a connected component of $\Gamma_S(\Ar_1)$. 
\begin{enumerate}
\item If $\mathfrak{C}$ contains the simple $\Ar_1$-module $\M[\1_{1_z}]$ corresponding to the vertex $1_z$ for some $z\in \Z$, then $\Omega_{\Ar_1} \mathfrak{C}$ contains the simple $\Ar_1$-module $\M[\1_{2_z}]$ corresponding to the vertex $2_z$. In this situation $\mathfrak{C}$ as well as $\Omega_{\Ar_1}\mathfrak{C}$ look as in Figure \ref{fig2} (left). 
\item If $\mathfrak{C}$ contains the string  $\Ar_1$-module $\M[\alpha_z]$ (resp. $\M[\beta_z]$) for some $z\in \Z$, then $\Omega_{\Ar_1} \mathfrak{C}$ contains the string $\Ar_1$-module $\M[\beta^\ast_z]$ (resp. $\M[\alpha^\ast_z]$). In this situation, $\mathfrak{C}$ as well as $\Omega_{\Ar_1} \mathfrak{C}$ look as in Figure \ref{fig2} (right). 
%\item If $\mathfrak{C}$ contains a band $\Ar_1$-module $\M[B, \lambda,n]$ with $B$ a band representative for $\Ar_1$, $\lambda \in \mathbb{C}^\ast$ and $n\in \Z$ with $n\geq 1$, then $\mathfrak{C}$ looks as in Figure \ref{fig4}. 
\end{enumerate}
\end{remark}

\begin{remark}\label{rem4.3}
Let $\mathfrak{C}$ be a connected component of $\Gamma_S(\Ar_1)$ containing string $\Ar_1$-modules. It follows from the description of the irreducible morphisms between string modules as in Definition \ref{defi4.2} that $\mathfrak{C}$ is completely determined by a string $\Ar_1$-module $\M[C]$, where $C$ is a string representative for $\Ar_1$ of minimal string length. Note also that the vertices and the arrows of $\Q$ induce string representatives of minimal string length.
\end{remark}

\begin{lemma}\label{lemma4.3}
Let $C$ be a string representative for $\Ar_1$ of minimal string length (as in Definition \ref{defi4.2} (iii)) that induces a string $\Ar_1$-module $\M[C]$. Then $C$ is string equivalent to the string induced by a vertex or by an arrow of $\Q$. 
\end{lemma}

\begin{proof}
If the string length of $C$ is $0$ or $1$, then there is nothing to prove. Assume then that $C$ has string length $\geq 2$. Then there is a string representative $C'$ for $\Ar_1$ such that $C$ is string equivalent to either $C'\lambda$ or $C'\sigma^{-1}$ for some arrows $\lambda, \sigma \in \Q_1$. Note that by assumption $C'$ has string length $\geq 1$. If $C'$ is string equivalent to an arrow of $\Q$ (i.e. $C'$ has string length $1$), then $C=(\1_v)_c$ or $C=(\1_v)_h$ for some vertex $v\in \Q_0$, which contradicts that $C$ is of minimal string length. Assume next that $C'$ has string length $\geq 2$. It follows that there exists a string representative $C''$ for $\Ar_1$ such that $C'$ is string equivalent to either $C''\gamma^{-1}$ or $C''\delta$ for some arrows $\gamma, \delta\in \Q_1$. Thus $C$ is string equivalent to either  
$C''\gamma^{-1}\lambda$ or $C''\delta\sigma^{-1}$, which implies that either $C=C''_c$ or $C=C''_h$, contradicting again that $C$ is of minimal string length. This finishes the proof of Lemma \ref{lemma4.3}.
\end{proof}

\begin{remark}\label{rem4.5}
In the following, we review the description of morphisms between string $\Ar_1$-modules as given in \cite{krause}. 
Let $S$ and $T$ be strings representatives for $\Ar_1$ and let $\M[S]$ and $\M[T]$ their respective string $\Ar_1$-modules. Suppose that $C$ is a substring of both $S$ and $T$ such that the following conditions (i) and (ii) are satisfied.
\begin{enumerate}
\item $S\sim S'CS''$, with ($S'$ of length zero or $S'=\hat{S}'\xi_1$) and ($S''$ of length zero or $S''=\xi_2^{-1}\hat{S}''$), where $S'$, $\hat{S}'$, $S''$, $\hat{S}''$ are strings and $\xi_1$, $\xi_2$ are arrows in $\Q$; and 
\item $T\sim T'CT''$, with ($T'$ of length zero or $T'=\hat{T}'\zeta_1^{-1}$) and ($T''$ of length zero or $T''=\zeta_2\hat{T}''$), where $T'$, $\hat{T}'$, $T''$, $\hat{T}''$ are strings and $\zeta_1$, $\zeta_2$ are arrows in $\Q$.
\end{enumerate}
Then by \cite{krause} there exists a composition of  $\Ar_1$-module homomorphisms
\begin{equation}\label{canhom}
\widehat{\sigma}_C:\M[S]\surjection \M[C]\injection \M[T], 
\end{equation}
where $\M[S]\surjection \M[C]$ denotes the canonical surjection and $\M[C]\injection \M[T]$ the canonical injection. We call $\sigma_C$ a {\it canonical homomorphism from $\M[S]$ to $\M[T]$}. Note that there may be several choices for $S'$, $S''$ (resp. $T'$, $T''$) in (i) (resp. (ii)). In other words, there may be several $\k$-linearly independent canonical homomorphisms factoring through $\M[C]$. By \cite{krause},  every $\Ar_1$-module homomorphism $\sigma:\M[S]\to \M[T]$ can be written as a unique $\k$-linear combination of canonical homomorphisms which factor through string modules corresponding to strings $C$ satisfying (i) and (ii). In particular, if $\M[S]=\M[T]$ then the canonical endomorphisms of $\M[S]$ generate $\End_{\Ar_1}(\M[S])$.

%In particular, if $S=T$, then every $\Ar_1$-module endomorphism of $\M[S]$ can be written as a unique $\k$-linear combination of the identity homomorphism and of canonical endomorphisms which factor through string $\Ar_1$-modules $\M[C]$ for suitable choices of $C$ satisfying (i) and (ii). 
%We call $\sigma_C$ a {\it canonical homomorphism} from $\M[S]$ to $\M[T]$ that factors through $\M[C]$. It follows from \cite{krause} that each $\Ar_1$-module homomorphism 
%from $M[S]$ to $M[T]$ can be written uniquely as a $\k$-linear combination of canonical $\Ar_1$-module homomorphisms as in (\ref{canhom}). In particular, if $M[S]=M[T]$ 
%then the canonical endomorphisms of $M[S]$ generate $\End_{\Ar_1}(M[S])$.
\end{remark}

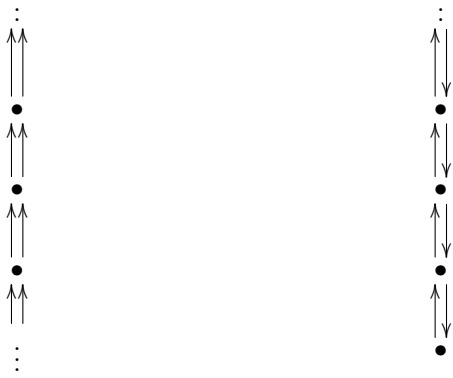
\begin{figure}
	\begin{align*}
	\xymatrix@=20pt{
       \vdots\\
       \bullet\ar@<0.5ex>[u]\ar@<-0.5ex>[u]\\
       \bullet\ar@<0.5ex>[u]\ar@<-0.5ex>[u]\\
       \bullet\ar@<0.5ex>[u]\ar@<-0.5ex>[u]\\
       \vdots\ar@<0.5ex>[u]\ar@<-0.5ex>[u]
      }
&&
\xymatrix@=20pt{
        \vdots\ar@<0.5ex>[d]\\
       \bullet\ar@<0.5ex>[u]\ar@<0.5ex>[d]\\
       \bullet\ar@<0.5ex>[u]\ar@<0.5ex>[d]\\
       \bullet\ar@<0.5ex>[u]\ar@<0.5ex>[d]\\
       \bullet\ar@<0.5ex>[u]
      } 
\end{align*}
	\caption{Components of $\Gamma_S(\Ar_1)$ containing a simple $\Ar_1$-module (left) and containing a string $\Ar_1$-module $\M[C]$ with $C$ an arrow of $\widehat{Q}$ (right) .}
	\label{fig2} 
\end{figure}

\begin{remark}\label{remexam}
For the remainder of this section, we need the following definition and property of morphisms between objects in $\widehat{\Ca}$. Following \cite{sch}, if $R$ is an object in $\widehat{\Ca}$, we denote by $t^\ast_R$ the quotient $\m_R/\m_R^2$ and call it the  {\it Zariski cotangent space} of $R$ over $\k$. Let $\theta: R\to R'$ be a morphism in $\widehat{\Ca}$.  It follows by \cite[Lemma 1.1]{sch} that $\theta$ is surjective if and only if the induced map of cotangent spaces $\theta^\ast: t^\ast_R\to t^\ast_{R'}$ is surjective. 
\end{remark}

%In the following, we prove Theorem \ref{thm2}, i.e., we classify the string $\Ar_1$-modules $\V$ such that $\SEnd_{\Ar_1}(\V)=\k$. In view of Theorem \ref{thm1}, if $\V$ is such a string module, then the versal deformation ring $R(\Ar_1,\V)$ is universal. Under this situation, we use the combinatorial description of $\V$ to determine the isomorphism class of $R(\Ar_1,\V)$. 

\begin{proof}[Proof of Theorem \ref{thm2}.]
Let $\V$ is a string $\Ar_1$-module and let $\mathfrak{C}$ be the connected component of $\Gamma_S(\Ar_1)$ containing $\V$. It follows from Remark \ref{rem4.3} and Lemma \ref{lemma4.3} that $\mathfrak{C}$ either contains a simple $\Ar_1$-module or contains a string $\Ar_1$-module $\M[C]$, where $C$ is an arrow of $\Q$.  This proves (i). 

Assume that $\mathfrak{C}$ contains a simple $\Ar_1$-module. Without loss of generality assume that $\mathfrak{C}$ contains $\M[\1_{1_z}]$ for a fixed $z\in \Z$. If $\V$ is a simple $\Ar_1$-module, then by Schur's Lemma, $\End_{\Ar_1}(\V)=\k$. Assume that $\V$ is not simple. Then $\V$ is isomorphic either to $\M[(\alpha_z^{-1}\beta_z)^n]$ or to $\M[(\beta_{z+1}^\ast(\alpha^\ast_{z+1})^{-1})^n]$ for some integer $n\geq 1$. It follows by Remark \ref{rem4.5} that in both situations the only canonical endomorphism of $\V$ as in (\ref{canhom}) is the identity morphism, and thus $\End_{\Ar_1}(\V)=\k$. On the other hand, from the description of the indecomposable projective $\Ar_1$-modules as in Figure \ref{fig1}, we obtain that $\Omega_{\Ar_1}\V$ is either $\M[\1_{2_z}]$, or $\M[(\beta_z\alpha_z^{-1})^m]$, or $\M[((\beta_z^\ast)^{-1}\alpha_z^\ast)^m]$, for some $m\geq 1$. It follows that for all these cases, $\Hom_{\Ar_1}(\Omega_{\Ar_1}\V,\V)$ is either zero or generated by canonical morphisms as in (\ref{canhom}) that factor through one of the simple $\Ar_1$-modules $\M[\1_{2_z}]$ or $\M[\1_{1_z}]$. Moreover, these canonical morphisms factor through one of the projective $\Ar$-modules $\P_{1_z}$ or $\P_{2_{z+1}}$.  Thus by Remark \ref{rem1.1} (vi) we have that $\Ext_{\Ar_1}^1(\V,\V)=\SHom_{\Ar_1}(\Omega_{\Ar_1}\V,\V)=0$, which together with Remark \ref{remark3} (i) implies that $R(\Ar_1,\V)\cong \k$. This proves (ii).

Next assume that $\mathfrak{C}$ contains a string $\Ar_1$-module $\M[C]$, where $C$ is an arrow of $\Q$. Without loss of generality, assume that $C=\alpha_z$ for a fixed $z\in \Z$. If $\V$ is isomorphic to $\M[\alpha_z]$, then again by Remark \ref{rem4.5}, it follows that $\SEnd_{\Ar_1}(\V)=\End_{\Ar_1}(\V)=\k$, which implies that the versal deformation ring $R(\Ar_1,\M[\alpha_z])$ is universal. Assume that $\V$ is not isomorphic to $\M[\alpha_z]$. Then $\V$ is isomorphic to $\M[(\alpha_z\beta_z^{-1})^n\alpha_z]$ for some integer $n\geq 1$. In this situation, there is a canonical endomorphism of $\V$ as in (\ref{canhom}) that factors through $\M[\alpha_z]$ and which does not factor through a projective $\Ar_1$-module. Therefore $\dim_\k \SEnd_{\Ar_1}(\V)\geq 2$. On the other hand, note that $\Hom_{\Ar_1}(\Omega_{\Ar_1}\M[\alpha_z], \M[\alpha_z])$ is generated by a canonical morphism as in (\ref{canhom}) that factors through the simple $\Ar_1$-module $\M[\1_{2_z}]$ and which does not factor through a projective $\Ar_1$-module. Thus by Remark \ref{rem1.1} (vi), we have that $\dim_\k\Ext_{\Ar_1}^1(\M[\alpha_z],\M[\alpha_z])=\dim_\k \SHom_{\Ar_1}(\Omega_{\Ar_1}\M[\alpha_z], \M[\alpha_z])=1$, which together with Remark \ref{remark3} (i) implies that $R(\Ar_1,\M[\alpha_z])$ is a quotient of $\k[\![t]\!]$.  

In the following, we use similar arguments as those in \cite[Claim 4.5]{calderon-giraldo-rueda-velez} to prove that $R(\Ar_1,\M[\alpha_z])$ is isomorphic to $\k[\![t]\!]$.

Let $\V_0=\M[\alpha_z]$, and for all $n\geq 1$, let $\V_n=\M[(\alpha_z\beta_z^{-1})^n\alpha_z]$. Thus for all $n\geq 1$, we obtain a non-trivial endomorphism of $\V_n$ factoring through $\V_{n-1}$ given by 
\begin{equation*}
\widehat{\sigma}_n: \V_n\surjection \V_{n-1}\injection \V_n.
\end{equation*}

Let $n\geq 1$ be fixed. Note that  the kernel of $\widehat{\sigma}_n$ as well as the image of $\widehat{\sigma}_n^n$ are isomorphic to $\V_0=\M[\alpha_z]$. This implies that the $\Ar_1$-module $\V_n$ is naturally a $\k[\![t]\!]/(t^{n+1})\Ar_1$-module by letting $t$ act on $x\in \V_n$ as $t\cdot x = \widehat{\sigma}_n(x)$. In particular $t\V_n\cong \V_{n-1}$. Assume next that  $\{\bar{r}_1, \bar{r}_2\}$ is a $\k$-basis of $\V_0=\M[\alpha_z]$. Using the 
isomorphism $\V_n/t\V_n\cong \V_0$, we can lift the elements $\bar{r}_1$ and $\bar{r}_2$ to corresponding elements $r_1, r_2 \in \V_n$. It follows that $\{r_1, r_2\}$ is linearly independent over $\k$ and that $\{t^sr_1, t^sr_2: 1\leq 
s\leq n\}$ is a $\k$-basis of $t\V_n\cong V_{n-1}$. Therefore, $\{r_1, r_2\}$ is a $\k[\![t]\!]/(t^{n+1})$-basis of $\V_n$, i.e.  $\V_n$ is free over $\k[\![t]\!]/(t^{n+1})$. On the other hand, observe that $\V_n$ lies in a short exact sequence of $\Ar_1$-modules 

\begin{equation*}
0\to t\V_n\to \V_n\to \k\otimes_{\k[\![t]\!]/(t^{n+1})}\V_n\to 0,
\end{equation*}
which implies that there exists an isomorphism of $\Ar_1$-modules $\widehat{\phi}_n:\k\otimes_{\k[\![t]\!]/(t^{n+1})}\V_n\to V_0$. Therefore $(\V_n,\widehat{\phi}_n)$ is a lift of $\V_0$ over $\k[\![t]\!]/(t^{n+1})$.

Note that for all $n\geq 1$, there are canonical projections $\pi_{n,n-1}: \V_n\to \V_{n-1}$. Let $\W=\invlim \V_n$ and let $t$ act on $\W$ as $\invlim \pi_{n,n-1}$. It follows that  $\W$ is a $\k[\![t]\!]\otimes_\k\Ar_1$-module and $\k
\otimes_{\k[\![t]\!]}\W\cong \W/t\W\cong \V_0$. This implies that there exists an isomorphism of $\Ar_1$-modules $\widehat{\varphi} :\k\otimes_{\k[\![t]\!]}\W\to \V_0$. After arguing as before, we obtain that $(\W,\widehat{\varphi})$ is a lift of $\V_0$ over $\k[\![t]\!]$. Thus there exists a unique $\k$-algebra homomorphism $\theta:R(\Ar_1,\V_0)\to \k[\![t]\!]$ in $\widehat{\Ca}$ corresponding to the deformation defined by $(\W,\widehat{\varphi})$. Note that since  $\W/t^2\W\cong \V_1$ as $\Ar_1$-modules, we obtain that $\W/t^2\W$ defines a non-trivial lift of $\V_0$ over $\k[\![t]\!]/(t^2)$ and thus there exists a unique morphism  $\theta': R(\Ar_1,\V_0)\to \k[\![t]\!]/(t^2)$. Note that since the cotangent space (as in Remark \ref{remexam}) of $\k[\![t]\!]/(t^2)$ has  dimension $1$ over $\k$, it follows that $\theta'$ is also surjective. Moreover, if $\theta'':\k[\![t]\!]\to \k[\![t]\!]/(t^2)$ is the canonical surjection,  then the uniqueness of $\theta'$ implies that $\theta'=\theta''\circ \theta$. Since the morphism $(\theta'')^\ast$ of contangent spaces (as in Remark \ref{remexam}) is an isomorphism of $\k$-vector spaces, it follows that $\theta^\ast$ is a surjection. Therefore by Remark \ref{remexam}, $\theta$ is surjective, which together with the fact that  $R(\Ar_1,\V_0)$ is a quotient of $\k[\![t]\!]$, implies that $\theta$ is an isomorphism. Hence $R(\Ar_1,\V_0)=R(\Ar_1,\M[\alpha_z])\cong \k[\![t]\!]$. This proves (iii), which finishes the proof of Theorem \ref{thm2}.
\end{proof}

\section{Acknowledgments}  
The fourth author would like to express his gratitude to the other authors, faculty members, staff and students at the Instituto of Matem\'aticas as well as to the other people related to this work at the 
Universidad de Antioquia for their hospitality and support during the developing of this project.  All the authors are grateful with the anonymous referee who provided many suggestions and corrections that improved the quality and the readability of this article, and who also pointed out major errors in a previous version of the statement and proof of Theorem \ref{thm2}. 

\section{Data availability statement}  
This manuscript has no associated data.

\bibliographystyle{amsplain}
\bibliography{Deformations_Repetitive_AlgebraRev2}

\end{document}